\documentclass{jgcc}
\pdfoutput=1
\usepackage{lastpage}
\jgccdoi{18}{1}{5}{17195}
\jgccheading{}{\pageref{LastPage}}{}{}{Dec.~26,~2025}{Jul.~7,~2026}{}

\usepackage{amsfonts,amssymb}
\usepackage{orcidlink}
\theoremstyle{definition}\newtheorem{defn}[thm]{Definition}
\renewcommand{\phi}{\varphi}

\newcommand{\N}{\mathbb N}

\newcommand{\R}{\mathbb R}

\newcommand{\Aut}{{\rm Aut}}
\newcommand{\Out}{{\rm Out}}

\def\epsilon{\varepsilon}
\def\phi{\varphi}

\def\hat{\widehat}

\newcommand{\cvr}{\mbox{cv}_r}

\begin{document}

\title[On the Hausdorff dimension and attracting laminations]{On the Hausdorff dimension and attracting laminations for fully irreducible automorphisms of free groups}

\author[I.~Kapovich]{Ilya Kapovich\,\texorpdfstring{\orcidlink{0000-0002-7694-6236}}{}}
\address{Department of Mathematics and Statistics, Hunter College of CUNY, 695 Park Ave, New York, NY 10065, U.S.A.}
\email{ik535@hunter.cuny.edu}

\dedicatory{Dedicated to Alexei Miasnikov}

\begin{abstract}
Motivated by a classic theorem of Birman and Series about the set of complete simple geodesics on a hyperbolic surface, we study the Hausdorff dimension of the set of endpoints in $\partial F_r$ of some abstract algebraic laminations associated with free group automorphisms.

For an exponentially growing outer automorphism $\phi\in\Out(F_r)$ we show that the set of endpoints $\mathcal E_{L}\subseteq \partial F_r$ of any of the \emph{attracting laminations} $L$ of $\phi$ has Hausdorff dimension $0$ and packing dimension $0$, for any visual metric on the boundary $\partial F_r$, and similarly that $L\subseteq \partial^2 F_r$ (where $\partial^2 F_r$ is equipped with the product metric of a visual metric) has Hausdorff dimension $0$ and packing dimension $0$. If $\phi\in\Out(F_r)$ is an atoroidal and fully irreducible, we deduce the same conclusions for the ending lamination $\Lambda_\phi$ of $\phi$ that gets collapsed by the Cannon-Thurston map $\partial F_r\to \partial G_\phi$ for the associated free-by-cyclic group $G_\phi=F_r\rtimes_\phi\mathbb Z$. By contrast, the set of endpoints of any of these laminations has upper box dimension $>0$ for any visual metric on $\partial F_r$.
\end{abstract}

\maketitle

\section{Introduction}

%\begin{thm}\label{t:1}
%This is a test theorem.
%\end{thm}

%In Theorem~\ref{t:1} we have similar to %\cite{An}

A classic 1985 result of Birman and Series~\cite{BirSer} shows that if $S$ is a closed hyperbolic surface then the union of all complete simple geodesics on $S$ is nowhere dense set of Hausdorff dimension $1$. Their result implies, in particular, that if $\lambda$ is a geodesic lamination on $S$ and $\tilde\lambda$ is the lift of $\lambda$ to $\mathbb H^2=\tilde S$, then the set $L$ of all $(p,q)$ in $\partial \mathbb H^2\times \partial \mathbb H^2 - diag=\mathbb S^1\times \mathbb S^1 - diag$ such that the geodesic from $p$ to $q$ in $\mathbb H^2$ belongs to $\tilde \lambda$, is nowhere dense in $\mathbb S^1\times \mathbb S^1 - diag$ and has Hausdorff dimension $0$. Moreover, since the projection from $\mathbb S^1\times \mathbb S^1 - diag$ to $\mathbb S^1$ is Lipschitz, and Lipschitz maps do not increase the Hausdorff dimension, it follows that $\{p\in \mathbb S^1 : (p,q)\in L \text{ for some } q\}$ has Hausdorff dimension $0$ as a subset of $\mathbb S^1$.

There have been many generalizations and extensions of the Birman-Series theorem; in particular see the work of Lenzhen and Souto~\cite{LS18} and of Sapir~\cite{S21}. It is interesting to understand what kind of analogs of the Birman-Series theorem hold for arbitrary word-hyperbolic groups. The notion of a simple geodesic usually no longer has a reasonable interpretation there, but for a non-elementary word-hyperbolic group $G$ one can still study "abstract algebraic laminations" on $G$, see Section~\ref{section:AAL} below. Namely, an \emph{abstract algebraic lamination} $L$ on $G$ is a subset $L\subseteq \partial ^2 G=\partial G\times \partial G - diag$ such that $L$ is closed, $G$-invariant and flip-invariant. Abstract algebraic laminations on word-hyperbolic groups naturally arise as supports of geodesic currents, as non-injectivity sets of Cannon-Thurston maps, and also come from some constructions associated with iterating automorphisms and with $\mathbb R$-tree actions. In this paper we consider the case of $G=F_r$, the free group of finite rank $r\ge 2$, which already provides a wide variety of interesting "non-classical" examples to consider. 

%Again, since the projection from $\partial^2 F_r$ to $\partial F_r$ is Lipschitz, if an abstract algebraic lamination $L$ on $F_r$ has Hausdorff dimension $0$ as a subset of $\partial^2 F_r$ implies that its projection $\mathcal E_L$ to $\partial F_r$ has Hausdorff dimension $0$ there.  

 For an arbitrary metric space $(X,d)$ and a subset $A\subseteq X$, one has $\dim_H(A)\le \dim_p(A)$, where $\dim_H(A)$ is the Hausdorff dimension of $A$ and where $\dim_p(A)$ is the \emph{packing dimension} of $A$. The packing dimension behaves  better than the Hausdorff dimension with respect to direct products, and we obtain our Hausdorff dimension results in this paper as consequences of proving that for the sets under consideration the packing dimension is equal to $0$.

Even in the free group case, in terms of looking for possible analogs of the Birman-Series theorem, it is necessary to limit the class of abstract algebraic laminations on $F_r$ under consideration to examples that have some similarities to "nice" geodesic laminations on hyperbolic surfaces. Otherwise one can easily get various (non-interesting) examples with positive Hausdorff dimension. E.g. the boundary $\partial CV_r$ of the projectivized Outer space $CV_r$ contains the (projective class of) the Bass-Serre tree $T$ corresponding to a proper free product decomposition $F_r=U\ast V$. If the free factor $U$ is non-cyclic, then the dual algebraic lamination $L(T)$ of $T$ (as defined in \cite{CHL08b}) contains a copy of $\partial^2 U$, which leads to positive Hausdorff dimension for $L(T)\subseteq \partial^2 F_r $ and for its set of endpoints in $\partial F_r$.

Our additional motivation comes from trying to understand how large, e.g. in terms its Hausdorff dimension, the set of non-conical limit points is for various convergence group actions of non-elementary word-hyperbolic groups on boundaries of Gromov-hyperbolic spaces, particularly those actions that come together with a Cannon-Thurston map (see more on this topic below).  This question is only partially understood even in the classical Kleinian groups context. A recent result of M. Kapovich and Liu~\cite{KL20} shows that if $\Gamma\le Isom(\mathbb H^3)$ is a finitely generated, non-free, torsion-free geometrically infinite Kleinian group such that the injectivity radius of $\mathbb H^3/\Gamma$ is bounded away from $0$, then the Hausdorff dimension of the non-conical limit set of $\Gamma$ is positive. This result applies, in particular, to the fiber surface subgroups of closed fibered hyperbolic 3-manifolds. However, the case of free geometrically infinite groups is not yet understood even in the classical Kleinian groups setting. Note, however, that for a f.g. geometrically infinite $\Gamma\le Isom(\mathbb H^3)$, such that the injectivity radius of $\mathbb H^3/\Gamma$ is bounded away from $0$ and the geometrically infinite ends are bounded by incompressible surfaces, the Hausdorff dimension of the non-horospherical limit set is known to be $0$ (Mahan Mj, private communication, via M. Kapovich). See the recent paper of  Lecure and Mj~\cite{LM18} for related results on non-horospherical limit points. The non-injectivity set of the Cannon-Thurston map, that is, the set of endpoints of the \emph{Cannon-Thurston algebraic lamination}, provides a nice canonical subset of the set of non-conical limit points (and of non-horospherical limit points). Thus it becomes interesting to try to prove that the set of endpoints of the Cannon-Thurston lamination has Hausdorff dimension $0$, and the case of the free group $F_r$ is particularly intriguing in light of the above discussion.

Recall that an element $\phi\in Out(F_r)$ is \emph{exponentially growing} if for some (equivalently, any) free basis $A$ of $F_r$ there exists $1\ne w\in F_r$ such that
\[
\lambda_A(\phi, w)=\limsup_{n\to\infty}\sqrt[n]{||\phi^n(w)||_A} >1.
\]
Here $||u||_A$ is the cyclically reduced length of $u\in F_r$ with respect to $A$. It is known~\cite{Le09} that in the above formula for an arbitrary $\phi\in Out(F_r)$ the actual limit $\lambda(\phi, w)\ge 1$ as $n\to\infty$ always exists and is independent of $A$. Moreover, for $\phi\in Out(F_r)$ there are only finitely many possible values for $\lambda(\phi, w)$ as $w$ varies over $F_r-\{1\}$, see \cite{Le09}. If $\phi\in Out(F_r)$ is fully irreducible then $\lambda(\phi)=\lambda(\phi, w)>1$, for all $w\ne 1$ in $F_r$ such that $[w]$ is a non-periodic conjugacy class for $\phi$.

In \cite{BFH00} Bestvina, Feighn and Handel, given any exponentially growing $\phi\in \Out(F_r)$ define a finite nonempty set of \emph{$\mathcal L(\phi)$ attracting laminations} for $\phi$. If $\phi$ is fully irreducible, the set $\mathcal L(\phi)$ consists of a single element, namely $\mathcal L(\phi)=\{L_\phi\}$, and $L_\phi$ is called \emph{the attracting lamination} of $\phi$. Thus for an exponentially growing $\phi\in \Out(F_r)$, we have $\mathcal L(\phi)=\{L_1,\dots, L_m\}$ and each $L_k=L_k(\phi)\subseteq \partial^2 F_r$ is an abstract algebraic lamination on $F_r$.

The technical definition of elements of $\mathcal L(\phi)$ in \cite{BFH00} in terms of relative train track representatives of $\phi$ (which is ultimately shown to be independent to be of the choice of such representatives) is rather complicated and we omit its details here. We explain how this definition works for fully irreducible $\phi\in \Out(F_r)$  in Section~\ref{sect:L} below and note that for an arbitrary exponentially growing $\phi\in \Out(F_r)$ elements of $\mathcal L(\phi)$ can be recovered from understanding fixed points in $\partial F_r$ of representatives of $\phi$ in $\Aut(F_r)$. The latter fact turns out to be sufficient for our purposes using a recent result of Hilion and Levitt~\cite{HL22}.

Attracting laminations of fully irreducibles play an important role in the study of the dynamics and geometry of free group automorphisms and appear naturally in many contexts, see for example \cite{BFH97,BFH00,KL10,KL15,DGT23}, etc.  More generally, attracting laminations of exponentially growing elements of $Out(F_r)$ proved to be a useful tool in the structure theory of $Out(F_r)$ and its subgroups, see ~\cite{BFH00,HM19,HM20,HM23}.

Recall that the (unprojectivized) \emph{Outer space} $\cvr$ consists of all $\mathbb R$-trees $T$ equipped with a minimal free discrete isometric action of $F_r$, where two such trees are considered equal if there exists an $F_r$-equivariant isometry between them. See Section~\ref{sect:OS} below for further discussion and background references.

When working with products of metric spaces (and their subsets), we take the product metric to be the supremum metric. That is, if $(X,d_X)$ and $(Y, d_Y)$ are two metric spaces, we equip $X\times Y$ with the \emph{product metric} $\hat d$ defined as $\hat d((x,y), (x',y'))=\max\{ d_X(x,x'), d_Y(y,y')\}$, where $x,x'\in X$ and $y,y'\in Y$. Most other natural choices of a metric on $X\times Y$ are bi-Lipschitz equivalent to $\hat d$, and hence they don't change the Hausdorff and packing dimensions of subsets of $(X\times Y. \hat d)$.

If $(X,d)$ is a Gromov-hyperbolic metric space with a basepoint $x_0\in X$ and if $a>1$, a metric $d$ on $\partial X$ is called a \emph{visual metric} with the \emph{visual parameter $a$} if there exists $C\ge 1$ such that for any $p,q\in \partial X$ 
\[
\frac{1}{C}a^{-(p,q)_{x_0}}\le d(p,q)\le C a^{-(p,q)_{x_0}}
\]
where $(p,q)_{x_0}$ is the Gromov prpduct of $p$ and $q$ with respect to $x_0$, see \cite{KB02} for precise definitions. If $X$ is an \emph{$\R$-tree} (that is a $0$-hyperbolic geodesic metric space) then $(p,q)_{x_0}=d_T(x_0,\tau)$ where $\tau$ is the bi-infinite geodesic from $p$ to $q$ in $T$; moreover, in this case $d_a=a^{-(p,q)_{x_0}}$ is a visual metric on $\partial X$ with the visual parameter $a$, for any $a>1$. This is the only explicit construction of a visual metric that we will need in this paper. For an arbitrary Gromov-hyperbolic geodesic metric space $(X,d)$ it is known that is some $a_0>1$ such that for every $1<a\le a_0$ there exists a visual metric on $\partial X$ with visual parameter $a$. Again, see \cite{KB02} for details. As noted above, when working with $\partial G\times \partial G$, we will usually equip it with the product metric $\hat d$ corresponding to some visual metric $d$ on $\partial G$. For subsets of $\partial G\times \partial G$ (including $\partial^2 G=\partial G\times \partial G - diag$) we will use the restriction of the product metric $\hat d$.

Our main result is (see Theorem~\ref{thm:HD-exp} below):
\begin{thm}\label{thm:A}
Let $\phi\in\Out(F_r)$ (wher $r\ge 2$) be exponentially growing and let $L\in \mathcal L(\phi)$, $L\subseteq \partial^2 F_r$ be an attracting lamination of $\phi$. Let $\mathcal E_{L}\subseteq \partial F_r$ be the set of endpoints of $L$.
 Let $a>1$ be an arbitrary visual parameter and let $d$ be a visual metric on $\partial F_r$ with visual parameter $a$, and let $\hat d$ be the corresponding product metric on $\partial F_r\times \partial F_r$.
 
 Then the metric spaces $(\partial F_r, d)$ and $(\partial^2 F_r, \hat d)$ we have

\[
\dim_H(\mathcal E_{L})=\dim_p(\mathcal E_{L})=0, \quad \text{ and } \dim_H(L)=\dim_p(L)=0
\]
where $\dim_H$ is the Hausdorff dimension and $\dim_p$ is the packing dimension.

\end{thm}

To prove Theorem~\ref{thm:A}, we establish a more general result, Theorem~\ref{thm:Dim} below (see also Corollary~\ref{cor:var} below for a more precise and general statement), which shows that the conclusion of Theorem~\ref{thm:A} holds for any abstract algebraic lamination $L$ on $F_r$ whose laminary language $L_\Gamma$, with respect to some (equivalently, every) marked graph $\Gamma$ in $\cvr$, is subexponentially growing in terms of its subword complexity (the latter condition is  equivalent to the laminary language $L_\Gamma$ of $L$ having topological entropy $h(L_\Gamma)=0$, see Remark~\ref{rem:ent} and Corollary~\ref{cor:var} below).

\begin{rem}
For a bounded metric space $S$ one has the general inequalities $\dim_H(S)\le \dim_p(S)\le \overline\dim_B(S)$ where $\overline\dim_B(S)$ is the \emph{upper box dimension} of $S$,  see Remark~\ref{rem:dimB} for the definition. As noted in Remark~\ref{rem:dimB} below, our proof of Theorem~\ref{thm:Dim} actually shows that standard "cylindrical slices" of $\mathcal E_L$ for a subexponentially growing algebraic lamination $L$ on $F_r$ have upper box dimension $0$. However, the upper box dimension is not well-behaved under countable unions. As again observed in Remark~\ref{rem:dimB}, since for a nonempty algebraic lamination $L$ the set of its endpoints $\mathcal E_L$ is dense in $\partial F_r$, it follows that $\overline\dim_B(\mathcal E_L)=\overline\dim_B(\partial F_r)$,  while  $\overline\dim_B(\partial F_r)>0$ for all visual metrics. Thus the conclusions of Theorem~\ref{thm:Dim} and Theorem~\ref{thm:A} cannot be extended to the upper box dimension.
\end{rem}

We also show, in Corollary~\ref{cor:equiv}, that the growth function of the language of an algebraic lamination on $L$ is an invariant (for a natural notion of equivalence) depending only on $L$ and not on the choice of a marked metric graph in $\cvr$. 
A result of Lustig~\cite{Lu22} implies this fact for free bases of $F_r$, that is for the case where the marked graphs are $r$-roses, and we provide an additional argument to cover the general case.

We first prove Theorem~\ref{thm:A} for the case where $\phi\in\Out(F_r)$ is fully irreducible (see  Theorem~\ref{thm:HD-fi}) using a more direct argument, using the fact that in this case the laminary language $(L_\phi)_\Gamma$ can be described via a primitive substitution (up to some technicalities associated with distinguishing between the cases where the attracting lamination $L_\phi$ is orientable or non-orientable, in the sense of \cite{DGT23}). The languages associated with primitive substitutions are known to be subexponentially growing by classic results from symbolic dynamics~\cite{LM21,Q}. We also show, see Corollary~\ref{cor:equiv} below, that, for an abstract algebraic lamination $L$ on $F_r$, the equivalence class of the growth function of the laminary language $L_\Gamma$ does not depend on the choice of a marked graph $\Gamma$.

For the case of an arbitrary exponentially growing $\phi\in\Out(F_r)$, we use a recent result of Hilion and Levitt~ \cite{HL22} mentioned above, who proved that in this case for every attracting lamination $L\in \mathcal L(\phi)$ the laminary language $L_\Gamma$ of $L$ is subexponentially growing (and in fact its complexity function $p_{L_\Gamma}(n)$ is at most quadratic). Using our Theorem~\ref{thm:Dim} this implies Theorem~\ref{thm:A}.

If $H\le G$ is a non-elementary word-hyperbolic subgroup of word-hyperbolic group $G$ and the inclusion $i: H\to G$ extends to a continuous map $\partial i:\partial H\to\partial G$, this extension $\partial i$ is called the \emph{Cannon-Thurston map}. The first nontrivial example of the Cannon-Thurston map was constructed by Cannon and Thurston~\cite{CT} in 1984 for fiber surface subgroups of closed hyperbolic $3$-manifold groups fibering over the circle. More generally, if $H$ is a non-elementary word-hyperbolic group $H$ with a convergence action on a metrizable compactum $Z$ and if this action extends to a continuous $H$-equivariant map $\partial i:\partial H\to Z$, this extension is also called the \emph{Cannon-Thurston map}. (In the previous setting one takes $Z=\partial G$.) Associated to $\partial i$ one can define $\Lambda=\{(p,q)\in \partial^2 H: \partial i(p)=\partial i(q)\}$. Then $\Lambda$ is an abstract algebraic lamination on $H$ called the \emph{Cannon-Thurston lamination}.

For $\phi\in \Out(F_r)$ put $G_\phi=F_r\rtimes_\phi \mathbb Z$ be the corresponding free-by-cyclic group. By \cite{Br00}, $G_\phi$ is word-hyperbolic if and only if $\phi$ is \emph{atoroidal}, that is $\phi$ has no nontrivial periodic conjugacy classes in $F_r$. If $G_\phi$ is word-hyperbolic, then a general result of Mitra~\cite{M98a} implies that the Cannon-Thurston map $\partial i:\partial F_r\to\partial G_\phi$ exists and is onto. Moreover, another general result of Mitra~\cite{M97} describes the fibers of $\partial i$ in this case. The quotient group $G_\phi/F_r=\mathbb Z$ has $\partial \mathbb Z=\{\pm \infty\}$. Associated to the two points comprising $\partial \mathbb Z$, there are \emph{ending algebraic laminations} $\Lambda_\phi, \Lambda_{\phi^{-1}}\subseteq \partial^2 F_r$ on $F_r$; see Section~\ref{sect:EL} below for their definition. The results of \cite{M97} imply that if $G_\phi$ is word-hyperbolic then for two distinct point $p,q\in \partial F_r$ one has $\partial i(p)=\partial i(q)$ if and only if $(p,q)\in \Lambda_\phi \cup \Lambda_{\phi^{-1}}$.  Thus the Cannon-Thurston lamination here is $\Lambda= \Lambda_\phi \cup \Lambda_{\phi^{-1}}$. In the case where $\phi\in \Out(F_r)$ is atoroidal and fully irreducible, Kapovich and Lustig~\cite{KL15} described the precise relationship between $L_\phi$ and $\Lambda_\phi$ and proved, in particular, that $\Lambda_\phi$ is equal to the transitive closure of $L_\phi$. Using this result and Theorem~\ref{thm:A} above, we obtain (see Corollary~\ref{cor:HD-fi} below):

\begin{cor}\label{cor:B}
Let $\phi\in \Out(F_r)$ is atoroidal and fully irreducible and let $G_\phi=F_r\rtimes_\phi \mathbb Z$ be the corresponding free-by-cyclic group. 

Let $\Lambda_\phi\subseteq \partial^2 F_r$ be the ending algebraic lamination of $\phi$. Let $\Lambda\subseteq \partial^2 F_r$ be the Cannon-Thurston lamination for the Cannon-Thurston map corresponding to the inclusion $F_r\le G_\phi$. Let $a>1$ be an arbitrary visual parameter and let $d$ be a visual metric on $\partial F_r$ with visual parameter $a$, and let $\hat d$ be the corresponding product metric on $\partial F_r\times\partial F_r$.

 Then for the metric spaces $(\partial F_r, d)$ and $(\partial F_r\times\partial F_r, \hat d)$  we have

\[
\dim_H(\mathcal E_{\Lambda_\phi})=\dim_p(\mathcal E_{\Lambda_\phi})=0 \text{ and  } \dim_H(\mathcal E_{\Lambda})=\dim_p(\mathcal E_{\Lambda})=0
\]
and 
\[
\dim_H(\Lambda_\phi)=\dim_p(\Lambda_\phi)=0 \text{ and  } \dim_H(\Lambda)= \dim_p(\Lambda)=0
\]

\end{cor}

In Section~\ref{sect:problems} we discuss several open problems raised by Theorem~\ref{thm:A} and Corollary~\ref{cor:B} and by the earlier related results.

{\bf Acknowledgements.} I am grateful to  Martin Lustig for helpful discussions, and to Gilbert Levitt for alerting me about the results of \cite{HL22}. I also thank the referee for helpful suggestions on improving the paper. I am especially grateful to my brother Michael Kapovich (1963-2026) whose questions provided the main inspiration and motivation for writing this paper.

The author was supported by the NSF grant DMS-1905641.

\medskip

{\bf Data availability statement.} We do not analyze or generate any datasets, because our work proceeds within a theoretical and mathematical approach.

\section{Symbolic dynamical systems}

We only give a brief overview of some basic facts and terminology from symbolic dynamics here and refer the reader to \cite{LM21,Q} for more detailed background information.

Throughout this paper we will assume that $\Sigma$ is a nonempty finite alphabet consisting of at least two elements.

\subsection{Shifts and subshifts}

The free monoid $\Sigma^\ast$ is the set of all finite words (including the empty word) over the alphabet $\Sigma$. For a finite word $v\in \Sigma^\ast$ we denote by $|v|$ the length of $v$. The alphabet $\Sigma$ is endowed with the discrete topology, and there are two natural associated \emph{full shift} spaces, the \emph{one-sided full shift} $\Sigma^\mathbb N$ (where $\mathbb N=\mathbb Z_{\ge 0}=\{0,1,2,3,\dots \}$) and  the \emph{two-sided full shift} $\Sigma^\mathbb Z$. Both $\Sigma^\mathbb N$ and $\Sigma^\mathbb Z$ are equipped with the product topology.  These spaces also come equipped with continuous \emph{shift maps} $s: \Sigma^\mathbb N\to \Sigma^\mathbb N$ and $s:\Sigma^\mathbb Z\to\Sigma^\mathbb Z$, defined by the same formula, $\big(s\left( (x_i)_i\right)\big)_j=x_{j+1}$. Closed $s$-invariant subsets of $\Sigma^\mathbb N$ and $\Sigma^\mathbb Z$ are called \emph{subshifts}.  A finite nonempty subword of a finite, semi-infinite or bi-infinite word $w$ over $\Sigma$ is called a \emph{factor}; the set of all factors of $w$ is denote $\mathcal F[w]$.

For a subshift $X$ of $\Sigma^\mathbb N$ or of $\Sigma^\mathbb Z$, the \emph{language} $\mathcal F[X]\subseteq \Sigma^\ast$ is $\mathcal F[X]:=\cup_{w\in X} \mathcal F[w]$.

\subsection{Substitutions}

For a finite nonempty alphabet $\Sigma$, a \emph{substitution} is a free monoid homomorphism $\theta:\Sigma^\ast\to\Sigma^\ast$ such that for every letter $a\in \Sigma$ we have $|\sigma(a)|>0$. Since $\Sigma^\ast$ is a free monoid, $\theta$ is uniquely determined by specifying a family of nonempty words $(\sigma(a))_{a\in \Sigma}$ in $\Sigma^\ast$.  A substitution $\theta:\Sigma^\ast\to\Sigma^\ast$ is called \emph{primitive} if there exists an integer $k\ge 1$ such that for all letters $a,b\in \Sigma$ the word $\theta^k(a)$ contains $b$. For a substitution $\theta:\Sigma\to\Sigma$ we denote by $\mathcal F[\theta]$ the set of all words $v\in \Sigma^\ast$ such that there exist $m\ge 1$ and $a\in \Sigma$ such that $v$ is a factor of $\theta^m(a)$; that is
\[
\mathcal F[\theta]=\cup_{m=1}^\infty\cup_{a\in \Sigma} \mathcal F[\theta^m(a)].
\]

Let $\theta:\Sigma^\ast\to\Sigma^\ast$ be a primitive substitution. There are two subshifts $X_\theta\subseteq \Sigma^\mathbb Z$ and $X_\theta^+\subseteq \Sigma^\mathbb N$ naturally associated with $\theta$. The subshift  $X_\theta\subseteq \Sigma^\mathbb Z$  is defined as the set of all $w\in \Sigma^\mathbb Z$ such that every factor of $w$ belongs to $\mathcal F[\theta]$. Similarly, $X_\theta^+\subseteq \Sigma^\mathbb Z$  is defined as the set of all $w\in \Sigma^\mathbb N$ such that every factor of $w$ belongs to $\mathcal F[\theta]$.  In both cases one has $\mathcal F[X_\theta]=\mathcal F[X_\theta^+]=\mathcal F[\theta]$.  The one-sided subshift $X_\theta^+$ also has a useful description in terms of periodic points of $\theta$. We say that $a\in \Sigma$ is a \emph{periodic direction} if there exists $k\ge 1$ such that $\theta^k(a)$ begins with $a$. The subset $\Sigma_0$ of $\Sigma$ of all periodic directions is nonempty. Choose $k\ge 1$ such that for every $a\in \Sigma_0$ for $\theta'=\theta^k$ we have $\theta'(a)$ starts with $a$. 
Then, the word $a$ is an initial segment of $\theta'(a)$, the word $\theta'(a)$ is an initial segment of $(\theta')^2(a)$, and so on. For each $a\in \Sigma_0$ we can define an \emph{eigenray} $\rho_a\in \Sigma^\mathbb N$ such that for all $m\ge 1$ the word $(\theta')^m(a)$ is an initial segment of $\rho_a$. By construction one has $\theta'(\rho_a)=\rho_a$ for all $a\in A$. It is then true (see \cite{Q}) that for every $a\in \Sigma$ we have $\mathcal F[\theta]=\mathcal F[\rho_a]$ and, moreover, the closure of the $s$-orbit of $\rho_a$ in $\Sigma^\mathbb N$ equals $X_\theta^+$. We will see further below that the attracting lamination for a fully irreducible outer automorphism $\phi\in \Out(F_r)$ is defined via a rather similar procedure to $X_\theta$ and $X_\theta^+$, using a train track map $f:\Gamma\to\Gamma$ representing $\phi$.

\subsection{Complexity functions}\label{sect:compl}

By a \emph{language} over a nonempty finite alphabet $\Sigma$ we mean a subset $\mathcal L\subseteq \Sigma^\ast$.

For a language $\mathcal L\subseteq \Sigma^\ast$ and an integer $n\ge 1$ we denote by $p_{\mathcal L}(n)$ the number of words $v\in \mathcal L$ with $|v|=n$ and we denote by $\beta_{\mathcal L}(n) $ the the number of nonempty words $v\in \mathcal L$ with $|v|\le n$. By convention we also set $p_{\mathcal L}(0)=\beta_{\mathcal L}(0)=0$.
We define the \emph{topological entropy} $h(\mathcal L)$ of $\mathcal L$ as
\[
h(\mathcal L)=\limsup_{n\to\infty}\frac{\log p_{\mathcal L}(n)}{n}.
\]

\begin{rem}\label{rem:c}
Note that if $X$ is a subshift  of $\Sigma^\mathbb N$ or of $\Sigma^\mathbb Z$, the language $\mathcal F[X]$ is closed under taking subwords and hence the complexity function $p_{\mathcal F[X]}(n)$ satisfies $p_{\mathcal F[X]}(n)\le p_{\mathcal F[X]}(n+1)$ for all $n\ge 1$. Moreover, it is known~\cite{LM21} that in this case $p_{\mathcal F[X]}(n)$ is submultiplicative, that is satisfies $p_{\mathcal F[X]}(n+m)\le p_{\mathcal F[X]}(n)p_{\mathcal F[X]}(m)$, where $m,n\ge 1$.
\end{rem}

For a subshift $X$ of $\Sigma^\mathbb N$ or of $\Sigma^\mathbb Z$ the \emph{topological entropy} $h(X)$ is defined as
\[
h(X)=\lim_{n\to\infty} \frac{\log p_{\mathcal F[X]}(n)}{n}.
\]
(It is known, in view of submultiplicativity of $p_{\mathcal F[X]}(n)$, that the above limit actually exists.) Note that by definition $h(X)=h(\mathcal F[X])$.

A key basic result from symbolic dynamics (see \cite[Proposition 5.12]{Q}) says:

\begin{prop}\label{prop:primsubst}
Let $\theta:\Sigma^\ast\to\Sigma^\ast$ be a primitive substitution. Then there exists $C>0$ such that for all $n\ge 1$
\[
p_{\mathcal F[\theta]}(n)\le Cn.
\]
\end{prop}

By summing up, the previous proposition directly implies:
\begin{cor}\label{cor:primsubst}
Let $\theta:\Sigma^\ast\to\Sigma^\ast$ be a primitive substitution. Then there exists $C'>0$ such that for all $n\ge 1$
\[
\beta_{\mathcal F[\theta]}(n)\le C'n^2.
\]

\end{cor}

Note that Proposition~\ref{prop:primsubst} implies that if $\theta:\Sigma^\ast\to\Sigma^\ast$ is a primitive substitution then $h(X_\theta)=h(X_\theta^+)=0.$

\begin{defn}[Subexponential function]
We will say that a monotone nondecreasing function $f:\mathbb N\to \N$ is \emph{subexponential} or \emph{subexponentially growing} if for every real number $a>1$ we have $\lim_{n\to\infty} \frac{f(n)}{a^n}=0$. 
\end{defn}

This condition has a useful equivalent reformulation:
\begin{lem}\label{lem:subexp}
Let $f:\mathbb N\to \N$ be a monotone nondecreasing function. Then $f$ is subexponential if and only if 
\[
\lim_{n\to\infty} \frac{\log f(n)}{n}=0.
\]
(Here we interpret $\log 0$ as $\log 0=0$.)
\end{lem}
\begin{proof}
(1) Suppose $f(n)$ is subexponential. If $f(n)=0$ for all $n\in N$, there is nothing to prove. Thus we assume that $f(n)\ge 1$ for all $n\ge n_0$ for some $n_0\ge 1$.

Then $f(n)\ge f(n_0)\ge 1$ for $n\ge n_0$, and hence for $n\ge n_0$ we have
\[
\frac{1}{n} \log f(n)\ge \frac{1}{n}\log n_0\to_{n\to\infty} 0,
\]
so that 
\[
\liminf_{n\to\infty} \frac{1}{n} \log f(n)\ge 0.
\]
Now fix any $\epsilon>0$ and put $a=e^\epsilon>1$.
By assumption
\[
\frac{f(n)}{e^{\epsilon n}}=\frac{f(n)}{a^{n}}\to_{n\to\infty} 0
\]
and therefore $f(n)\le e^{\epsilon n}$ for all large enough $n$.
Hence
\[
\frac{1}{n}\log f(n)\le \epsilon
\]
for all large enough $n$, and so
\[
\limsup_{n\to\infty} \frac{1}{n} \log f(n)\le \epsilon.
\]
Since $\epsilon>0$ was arbitrary, we have 
\[
\limsup_{n\to\infty} \frac{1}{n} \log f(n)\le 0.
\]
It follows that 
\[
\lim_{n\to\infty} \frac{\log f(n)}{n}=0,
\]
as required.

(2) Suppose now that
\[
\lim_{n\to\infty} \frac{\log f(n)}{n}=0.
\]

Choose $a>1$. Thus $\log a>0$. The above limit implies that for all large enough $n$ we have
\[
\frac{\log f(n)}{n}<\frac{1}{2}\log a.
\]
Therefore for $n\to\infty$ we have
\[
\log f(n)<\frac{1}{2} n\log a
\]
and 
\[
0\le f(n)<a^{n/2}.
\]
Then for $n\to\infty$
\[
0\le \frac{f(n)}{a^n}<a^{-n/2}\to_{n\to\infty} 0,
\]
which implies 
\[
\lim_{n\to\infty}\frac{f(n)}{a^n}=0.
\]
Since $a>1$ was arbitrary, the function $f(n)$ is subexponential, as required.
\end{proof}
\begin{defn}
We say that a language $\mathcal L\subseteq \Sigma^\ast$ is \emph{subexponentially growing} if the function $\beta_\mathcal L(n)$ is subexponential. Similarly, we say that a subshift $X$ is  \emph{subexponentially growing} if the language $\mathcal F[X]$ is subexponentially growing.
\end{defn}

\begin{rem}\label{rem:ent}
Note that, for a subshift $X$, $p_{\mathcal F[X]}(n)$ is monotone non-decreasing and $\beta_{\mathcal F[X]}(n)=\sum_{j=1}^n p_{\mathcal F[X]}(n)\le np_{\mathcal F[X]}(n)$ for $n\ge 1$. Thus we have $p_{\mathcal F[X]}(n) \le \beta_{\mathcal F[X]}(n)\le n p_{\mathcal F[X]}(n)$. Hence, in view of Lemma~\ref{lem:subexp}, we have $X$ is subexponentially growing if and only if $p_{\mathcal F[X]}(n)$ is subexponential if and only if $h(X)=0$ if and only if $h(\mathcal F[X])=0$.
\end{rem}

Corollary~\ref{cor:primsubst} implies that for a primitive substitution $\theta:\Sigma^\ast\to\Sigma^\ast$ the language $\mathcal F[\theta]$ is subexponentially growing, as are the subshifts $X_\theta$ and $X_\theta^+$.

\section{Abstract algebraic laminations}\label{section:AAL}

We refer the reader to \cite{KB02} for background info on the boundaries of word-hyperbolic groups. For the purposes of this paper the main cases of relevance are where $G$ is a free group or a surface group.

Let $G$ be a non-elementary word-hyperbolic group and let $\partial G$ be its hyperbolic boundary. We put $\partial^2 G=\partial G\times \partial G -diag =\{(p,q)\in \partial G\times \partial G: p\ne q\}$. We endow $\partial^2 G\subseteq \partial G\times \partial G$ with the subspace topology from $\partial G\times \partial G$. Note that $G$ has a natural translation action by homeomorphisms on $\partial G\times \partial G$ which leaves $\partial^2 G$ invariant and thus also gives an action of $G$ by homeomorphisms on $\partial^2G$: for $g\in G$ and $(p,q)\in \partial^2 G$ we put $g(p,q)=(gp,gq)$. The space $\partial^2 G$ also comes with a natural "flip map" $\mathfrak f: \partial^2G\to\partial G$ interchanging the coordinates, where $\mathfrak f: (p,q)\mapsto (q,p)$ for $(p,q)\in \partial^2G$.

\begin{defn}[Abstract algebraic lamination]
Let $G$ be a non-elementary word-hyperbolic group. An \emph{abstract algebraic lamination} on $G$ is a subset $L\subseteq \partial^2 G$ such that $L$ is closed in $\partial^2 G$, is $G$-invariant and $\mathfrak f$-invariant.
For an abstract algebraic lamination $L$ on $G$ a pair $(p,q)\in L$ (where $p,q\in \partial G$ and $p\ne q$) is called a \emph{leaf} of $L$ or an \emph{abstract leaf} of $L$.

We also denote $\mathcal E_L=\{p\in \partial G: (p,q)\in L \text{ for some } q\in \partial G\}$ and call $\mathcal E_L\subseteq \partial G$  \emph{the set of endpoints of $L$}.
\end{defn}

Suppose that a non-elementary word-hyperbolic group $G$ and that $G$ acts properly discontinuously and cocompactly by isometries on a proper Gromov-hyperbolic geodesic metric space $(X,d)$.
By the Milnor-Svarc lemma, the orbit map $G\to X, g\mapsto gx_0$ (where $x_0\in X$ is some base-point) defines a $G$-equivariant quiasi-isometry, which then extends to a $G$-equivariant homeomorphism $\partial G\to \partial X$. In this situation we will identify $\partial X$ with $\partial X$ via this homeomorphism and will usually suppress this identification.

Then any two distinct points in $\partial G$ can be connected by a (generally non-unique) bi-infinite geodesic in $X$. In this setting, if $L$ is an abstract geodesic lamination on $L$ and $\gamma$ is a bi-infinite geodesic in $X$ from $p\in \partial X$ to $q\in \partial X$ such that $(p,q)\in L$, we will say that $\gamma$ is a \emph{geometric leaf} of $L$ with respect to $X$.

Note that if $\Lambda$ is a geodesic lamination on a closed hyperbolic surface $S$ then $\Lambda$ can be viewed as an abstract geodesic lamination on $G=\pi_1(S)$ in the above sense. Lifts of the leaves of $\Lambda$ to $\tilde S=\mathbb H^2$ are geometric leaves with respect to $\mathbb H^2$ for this abstract geodesic lamination.

\section{Outer space}\label{sect:OS}

For the remainder of this paper let $F_r$ be the free group of finite rank $r\ge 2$. By that we mean that $F_r=F(a_1,\dots,a_r)$ is a free group with a particular preferred free basis $A=\{a_1,\dots, a_r\}$, although we usually will suppress the reference to this basis $A$.  We will briefly recall some of the Outer space terminology and notations here and refer the reader to \cite{B,CV,FM,V15} for a more detailed background info.

\subsection{Graphs}

We will use here the same conventions and notations regarding graphs and metric graphs that are carefully set up in \cite{DKL15}. We briefly recall some of the relevant terminology here and refer the reader to \cite{DKL15} for more details. By a \emph{graph} we mean a 1-dimensional CW-complex. For a graph $\Delta$ the set of $0$-cells of $\Delta$ is denoted by $V\Delta$ and its elements are called the \emph{vertices} of $\Delta$.  The set of non-oriented open topological 1-cells of $\Delta$ is denoted $E_{top}\Delta$ and its elements are called \emph{topological edges}. Every topological edge of $\Delta$ is a copy of the open unit interval and thus admits two exactly orientations (e.g. as a 1-manifold). An \emph{oriented edge} of $\Delta$ is a topological edge of $\Delta$ with a choice of an orientation on it. The set of all oriented edges of $\Delta$ is denoted $E\Delta$. There is a natural fixed-point-free involution ${}^{-1}:E\Delta\to E\Delta, e\mapsto e^{-1}$ corresponding to reversing the orientation on an oriented edge. Additionally, the attaching maps for $\Delta$ naturally define the \emph{enpoint maps} $o:E\Delta\to V$ and $t:E\Delta\to V$, where for $e\in E\Delta$, the vertex $o(e)$ is the \emph{origin} of $e$ and $t(e)$ is the \emph{terminus} of $e$. [Note that we allow $o(e)=t(e)$.] By construction, for every $e\in E\Delta$ we have $o(e^{-1})=t(e)$. For a vertex $v\in V\Delta$ the \emph{degree} $\deg_\Delta(v)$ is the number of $e\in E\Delta$ with $o(e)=v$. A \emph{metric graph} structure $\mathfrak J$ on $\Delta$ identifies every topological edge $e$ of $\Delta$ via a homeomorphism with a nonempty finite open interval $J\subseteq \mathbb R$ in such a way that this homeomorphism can be continuously extended to a map from the closure $\bar J$ of $J$ to the closure of $e$ in $\Delta$. In particular, this identification assigns every topological edge $e$ of $\Delta$ (and hence every oriented edge of $\Delta$) some new positive \emph{length} $\mathfrak J(e)$, namely the length of the interval $J$. For the most part, when working with metric graph structures, one only needs to use the information about these new edge-lengths. A \emph{metric graph} is a graph $\Delta$ together with a metric graph structure on it.

An \emph{orientation} on a graph $\Delta$ is a partition $E\Delta=E^+\Delta\sqcup E^-\Delta$ such that for every $e\in E\Delta$ we have $e\in E^+\Delta \iff e^{-1}\in E^-\Delta$.

\subsection{Paths}

For a graph $\Delta$, a \emph{path} or \emph{edge-path} in $\Delta$ is a sequence $\gamma=e_1,\dots,e_n$ of elements of $E\Delta$ such that for all $1\le i<n$ we have $t(e_i)=o(e_{i+1})$. In this case we say that $n=|\gamma|$ is the \emph{combinatorial length} of $\gamma$ and we put $o(\gamma):=o(e_1)$ and $t(\gamma):=t(e_n)$.  We also define the \emph{inverse path} $\gamma^{-1}:=e_n^{-1},\dots, e_1^{-1}$. If $\Delta$ is a metric graph with a metric graph structure $\mathfrak L$, then for an edge-path $\gamma$ as above, the \emph{metric length} of $\gamma$ is $\mathfrak J(\gamma)=\sum_{i=1}^n \mathfrak J(e_i)$. For any vertex $v\in V\Delta$ we also view $\gamma=v$ as a path with $o(\gamma)=t(\gamma)=v$ and $|\gamma|=\mathfrak J(\gamma)=0$, and we put $\gamma^{-1}=\gamma$. A path $\gamma$ in $\Delta$ is \emph{reduced} or \emph{immersed} if $\gamma$ contains no subpaths of the form $e, e^{-1}$ where $e\in E\Delta$.  The notions of an edge-path and of being reduced  naturally extend to semi-infinite and bi-infinite edge-paths in $\Delta$. We denote by $\Omega(\Delta)$ the set of all finite edge-paths in $\Delta$. For a subset $P\subseteq \Omega(\Delta)$, denote $\overline P=\{\gamma^{-1}:\gamma\in P\}$.

\subsection{Outer space}

The  \emph{unprojectivized Outer space} $\cvr$ is the space of all $\mathbb R$-trees $T$ with a minimal nontrivial free discrete isometric action of $F_r$, considered up to $F_r$-equivariant isometry.
Here "minimal" means that $T$ has no $F_r$-invariant subtrees $T'\subsetneq T$, and "nontrivial" means that $F_r$ does not have a global fixed point in $T$.
There is a natural topology as well as an action of $\Out(F_r)$ on $\cvr$, but their definitions are not important for the purposes of this paper. By abuse of notation, in the $\cvr$ setting, we will not distinguish between an $\mathbb R$-tree $T$ as above an its $F_r$-equivariant isometry class and will write $T\in\cvr$.

Note that if $T\in \cvr$ then the action of $F_r$ on $T$ is properly discontinuous and cocompact, and the quotient space $T/F_r$ is a finite topological graph $\Gamma$ with $\pi_1(\Gamma)=F_r$ where all vertices of $\Gamma$ have degrees $\ge 3$. Moreover, $\Gamma$ inherits the \emph{metric graph structure} from $T$ where we give each edge of $\Gamma$ the same length as any of its preimages in $T$.
Also, starting with a a finite connected metric graph with all vertices of degree $\ge 3$ and with $\pi_1(\Gamma)=F_r$, we get a point $T\in \cvr$ by using the universal cover $\tilde \Gamma$ with the lifted metric graph structure. In the case where such a graph $\Gamma$ has a single vertex, $\Gamma$ is an $r$-rose. If every petal of this rose is given length $1$, then $T$ is isometric to the Cayley tree of $F_r$ with respect to a free basis given by the petals of the $r$-rose.

In the above discussion, when we write $\pi_1(\Gamma)=F_r$, we mean that there is a particular isomorphism $\alpha: F_r\to \pi_1(\Gamma)$ identifying our fixed free group $F_r=F(a_1,\dots, a_r)$ with $\pi_1(\Gamma)$. This isomorphism, sometimes called a "marking", is an essential component of describing points of $\cvr$ when they are represented by metric graphs $\Gamma$. In such a description of points of $\cvr$ we always need a pair $(\Gamma, \alpha)$ where $\Gamma$ is a finite connected metric graph with all vertices of degree $\ge 3$ and with $b_1(\Gamma)=r$ and where $\alpha: F_r\to \pi_1(\Gamma)$ is a marking. We refer to such a pair $(\Gamma, \alpha)$ as a \emph{marked metric graph}.

Recall that by convention from Section~\ref{section:AAL}, every $T\in \cvr$ determines an identification $\partial F_r=\partial T$.
We also need to recall the construction of visual metrics on $\partial T$. If $X_0\in T$ and $p,q\in \partial T$,  we denote by $(p,q)_{x_0}$ the length of the common overlap in $T$ of the geodesic rays $[x_0,p)$ and $[x_0,q)$. Note that $(p,q)_{x_0}=\infty$ iff $p=q$ for $p,q\in \partial T$. For an arbitrary \emph{visual parameter} $a>1$ and a base-point $x_0\in T$, define the \emph{standard visual metric} $d_a$ on $\partial T=\partial G$ as
\[
d_a(p,q)=a^{-(p,q)_{x_0}},
\]
where $p,q\in \partial T$. It is well-known and easy to check that, since $T$ is an $\mathbb R$-tree, this $d_a$ is indeed a metric.

\section{Laminary languages}

\begin{defn}[Laminary language]\label{defn:LL}
Let $L\subseteq \partial^2 F_r$ be an abstract algebraic lamination on $F_r$. Let $T\in\cvr$ and let $\Gamma=T/F_r$ be the quotient marked metric graph with $\pi_1(\Gamma)=F_r$. For the finite alphabet $E\Gamma$ consider the language $L_\Gamma\subseteq (E\Gamma)^\ast$ consisting of all finite edge-paths $\gamma$ in $\Gamma$ with $|\gamma|>0$ such that there exists a bi-infinite geodesic $\tau$ in $T$ from $p\in \partial T$ to $q\in \partial T$ with $(p,q)\in L$ and such that some finite subpath $\tilde \gamma$ of $\tau$ is a lift of $\gamma$ to $T$.  We call $L_\Gamma\subseteq (E\Gamma)^\ast$ the \emph{laminary language of $L$ with respect to $\Gamma$.}
\end{defn}

Laminary languages for abstract algebraic laminations on free groups were formally introduced and studied in \cite{CHL08a}. All the definitions and results stated there deal with the case where $\Gamma$ is an $r$-rose (that is $\Gamma$ corresponds to some free basis of $F_r$). However, all of the discussion regarding laminary languages in \cite{CHL08a} works essentially verbatim, or with very minor changes, for arbitrary marked metric graphs $\Gamma$ as above.

\begin{defn}[Laminary subshifts]
Let $L$ be an abstract algebraic lamination on $F_r$, let $T\in \cvr$ and let $\Gamma=T/F_r$ be be the quotient marked metric graph with $\pi_1(\Gamma)=F_r$. Denote by $X_{L,\Gamma}$ the set of all bi-infinite reduced paths $w= \dots e_{-2}, e_{-1}, e_0,e_1,e_2,\dots $ in $\Gamma$ such that $w$ lifts to some bi-infinite geodesic in $T$ connecting $p\in \partial T$ to $q\in \partial T$ with $(p,q)\in L$. Denote by  and by $X_{L,\Gamma}^+$ the set of all semi-infinite reduced edge-paths $w_+$ in $\Gamma$ (with edges in $w_+$ indexed by $\N$) that appear as subpaths of bi-infinite paths from $X_{L,\Gamma}$ . Then $X_{L,\Gamma}\subseteq (E\Gamma)^\mathbb Z$ is a subshift with $\mathcal F[X_{L,\Gamma}]=L_\Gamma$. Similarly, $X_{L,\Gamma}^+\subseteq (E\Gamma)^\mathbb N$ is a subshift with $\mathcal F[X_{L,\Gamma}^+]=L_\Gamma$.
We call $X_{L,\Gamma}$ the \emph{laminary subshift} for $L$ with respect to the marked graph $\Gamma$. We call $X_{L,\Gamma}^+$ the \emph{one-sided laminary subshift} for $L$ with respect to $\Gamma$.
\end{defn}

Note that since $\mathcal F[X_{L,\Gamma}]=L_\Gamma$ we have $p_{L_\Gamma}(n)=p_{X_{L,\Gamma}}(n)$ for all $n\ge 1$.  Since $X_{L,\Gamma}$, Remark~\ref{rem:c} applies to $p_{L_\Gamma}(n)$. In particular, the function $p_{L_\Gamma}(n)$ is monotone non-decreasing and submultiplicative.

By convention, we denote $\mathbb N=\mathbb Z_{\ge 0}=\{m\in \mathbb Z: m\ge 0\}$.

\begin{defn}
For monotone non-decreasing functions $f,g:\mathbb N\to \mathbb N$, we say that $f\sim g$ if there exists $C\ge 1$ such that for all $n\ge 1$ we have
\[
f(n)\le C g(Cn), \quad \text{ and } g(n)\le Cf(Cn).
\]
Note that $\sim$ is an equivalence relation on the set of all monotone non-decreasing functions $\mathbb N\to \mathbb N$.
\end{defn}

Recall that in Section~\ref{sect:compl}, for a language $\mathcal L\subseteq \Sigma^\ast$ we defined the functions $p_\mathcal L(n)$ and $\beta_\mathcal L(n)$ for $n\ge 1$.  In the context of the laminary languages, by definition, for an integer $n\ge 1$ we have $p_{L_\Gamma}(n)=\#\{\gamma\in L_\Gamma: |\gamma|=n\}$ and $\beta_{L_\Gamma}(n)=\#\{\gamma\in L_\Gamma: |\gamma|\le n\}$. If $\Gamma$ is a metric graph with a metric structure $\mathfrak J$, we will extend the above definition of  $\beta_{L_\Gamma}(n)$ to this metric graph setting. Namely, for an integer $n\ge 1$ put
\[
\beta_{L_\Gamma, \mathfrak J}(n)=\#\{\gamma\in L_\Gamma: \mathfrak J(\gamma) \le n\}.
\]

\begin{lem}\label{lem:comp}
Let $L$ be an abstract algebraic lamination on $F_r$. Let $\Gamma=T/F_r$ be the quotient marked metric graph with the metric structure $\mathfrak J$. Then there exists $C\ge 1$ such that for all $n\ge 1$ we have
\[
\beta_{L_\Gamma, \mathfrak J}(n)\sim \beta_{L_\Gamma}(n).
\]
\end{lem}

\begin{proof}
Since $\Gamma$ is a finite graph, there exists $C\ge 1$ such that for every edge-path $\gamma$ in $\Gamma$ we have
\[
\frac{1}{C} |\gamma|\le \mathfrak J(\gamma)\le C|\gamma|.
\]
The conclusion of the lemma now follows.
\end{proof}

The following statement is essentially due to Lustig~\cite[Proposition~4.1(1)]{Lu22}:
\begin{prop}\label{prop:Lu}
Let $L$ be an abstract algebraic lamination on $F_r$. Let $T_1,T_2\in \cvr$ and let $\Gamma_1=T_1/F_r$, $\Gamma_2=T_2/F_r$ be the quotient marked graphs such that both $\Gamma_1$ and $\Gamma_2$ are $r$-roses. Let $L_{\Gamma_1}$ and $L_{\Gamma_2}$ be the corresponding laminary languages for $L$ with respect to $\Gamma_1$ and $\Gamma_2$. Then
\[
p_{L_{\Gamma_1}}(n)  \sim p_{L_{\Gamma_2}}(n).
\]
\end{prop}

In Proposition~\ref{prop:Lu} the marked graphs $\Gamma_1$ and $\Gamma_2$ can be viewed as two free bases of $F_r$, and Proposition~4.1(1) is provide in \cite{Lu22} for that setting. We need the conclusion of Proposition~\ref{prop:Lu} for arbitrary marked graphs coming from points of $\cvr$. Unfortunately, the proof of Proposition~4.1(1) given in \cite{Lu22} does not straightforwardly generalize to the case of arbitrary marked graphs coming from two pints of $\cvr$ because for marked graphs with more than one vertex there is no canonical procedure to rewrite a reduced edge-path in the first graph to a reduced edge-path in the second graph. Therefore we need an additional argument to help cover the general case:

\begin{prop}\label{prop:Y}
Let $L$ be an abstract algebraic lamination on $F_r$. Let $T\in \cvr$ and let $\Gamma=T/F_r$ be the quotient marked graph. Let $Y\subseteq \Gamma$ be a maximal subtree and let $\Gamma'$ be the $r$-rose obtained by collapsing $Y$ in $\Gamma$ to a single vertex.
Let $L_{\Gamma}$ and $L_{\Gamma'}$ be the corresponding laminary languages for $L$ with respect to $\Gamma$ and $\Gamma'$. Then
\[
p_{L_{\Gamma}}(n)  \sim p_{L_{\Gamma'}}(n).
\]

\end{prop}
\begin{proof}
We will assume that $Y$ is not a single vertex since otherwise $\Gamma=\Gamma'$ and there is nothing to prove.

Put $D\ge 1$ to be the diameter of the finite tree $Y$.

We give each edge of $\Gamma'$ length $1$ and put $T'=\widetilde{\Gamma'}\in\cvr$. Let $p,q\in \partial F_r$, $p\ne q$. Let $\tau$ be a bi-infinite geodesic from $p$ to $q$ in $T$ projecting to a bi-infinite reduced edge-path $\gamma$ in $\Gamma$, and let $\tau'$ be a bi-infinite geodesic from $p$ to $q$ in $T'$ projecting  to a bi-infinite reduced edge-path $\gamma'$ in $\Gamma'$. To obtain $\gamma'$ from $\gamma$ we have to delete all occurrences of edges of $Y$ from $\gamma$. Conversely, to recover $\gamma$ from $\gamma'$, for any two consecutive edges $e_{i}, e_{i+1}$ from $E\Gamma'=E(\Gamma-Y)$ in $\gamma'$ we have to insert the $Y$-geodesic path $[t_\Gamma(e_i),o_\Gamma(e_{i+1})]_Y$  and get $\gamma = \dots e_i [t_\Gamma(e_i),o_\Gamma(e_{i+1})]_Ye_{i+1}\dots $ between them.
Note that the $Y$-geodesic path $[t_\Gamma(e_i),o_\Gamma(e_{i+1})]_Y$ always has length $\le D$ here.

We therefore define the functions $f: L_{\Gamma}\to L_{\Gamma'}$ and $g:L_{\Gamma'}\to L_{\Gamma}$ as follows.

For a (finite) edge-path $\gamma\in L_{\Gamma}$, the path $f(\gamma)$ is obtained by deleting all edges of $Y$ from $\gamma$. For a (finite) edge-path $\gamma'\in L_{\Gamma}$, the path $g(\gamma')$ in $\Gamma$ is obtained from $\gamma'$ by inserting, between any two consecutive edges $e_{i}, e_{i+1}$ from $E\Gamma'=E(\Gamma-Y)$ in $\gamma'$, the $Y$-geodesic path $[t_\Gamma(e_i),o_\Gamma(e_{i+1})]_Y$.

The function $g$ is injective by construction and, moreover, for every $\gamma\in L_{\Gamma'}$, we have $|g(\gamma')|\le D|\gamma'|$.  Hence
\[
p_{L_{\Gamma'}}(n)\le p_{L_\Gamma}(Dn)
\]
for all $n\ge 1$.

For the function $f:L_{\Gamma}\to L_{\Gamma'}$ there exists an integer constant $C_0\ge 1$ such that for every element of $L_\Gamma'$ its full preimage under $f$ consists of $\le C_0$ elements. Namely, this bounded multiplicity comes from the finite number of choices in erasing the initial segment of length $\le D$ from $Y$ (before the first edge of $\Gamma-Y$) and in erasing the terminal segment of length $\le D$ from $Y$ (after the last edge of $\Gamma-Y$) from a path in $L_\Gamma$.  Also, by construction $|f(\gamma)|\le |\gamma|$ for all $\gamma\in L_\Gamma$. Therefore
\[
p_{L_\Gamma}(n)\le C_0 p_{L_\Gamma'}(n)
\]
for all $n\ge 1$.
The statement of the proposition now follows.
\end{proof}

\begin{prop}\label{prop:general}
Let $L$ be an abstract algebraic lamination on $F_r$. Let $T_1,T_2\in \cvr$ and let $\Gamma_1=T_1/F_r, \Gamma_2=T_2/F_r$ be the quotient marked graphs. Let $L_{\Gamma_1}$ and $L_{\Gamma_2}$ be the corresponding laminary languages for $L$.

Then
\[
p_{L_{\Gamma_1}}(n)  \sim p_{L_{\Gamma_2}}(n).
\]
\end{prop}
\begin{proof}
For $i=1,2$ let $\Gamma_i'$ be the marked $r$-rose obtained by collapsing maximal subtree $Y_i$ in $\Gamma_i$.  By applying Proposition~\ref{prop:Y} to each pair $\Gamma_i,\Gamma_i'$ for $i=1$ and applying Proposition~\ref{prop:Lu} to the pair $\Gamma_1',\Gamma_2'$, the statement of the proposition follows.
\end{proof}

\begin{cor}\label{cor:equiv}
Let $L$ be an abstract algebraic lamination on $F_r$. Let $T_1,T_2\in \cvr$ and let $\Gamma_1=T_1/F_r$, $\Gamma_2=T_2/F_r$ be the quotient marked metric graphs with the induced metric structures $\mathfrak J_1$ and $\mathfrak J_2$ accordingly. Then the following hold:
\begin{enumerate}
\item We have
\[
\beta_{L_{\Gamma_1}}(n)  \sim \beta_{L_{\Gamma_2}}(n).
\]

\item We have
\[
 \beta_{L_{\Gamma_1},\mathfrak J_1}(n)  \sim \beta_{L_{\Gamma_2},\mathfrak J_2}(n).
\]

\end{enumerate}
\end{cor}
\begin{proof}
Part (1) follows from Proposition~\ref{prop:general} by summing up. Part (2) then follows from part (1) together with Lemma~\ref{lem:comp}.
\end{proof}

Corollary~\ref{cor:equiv} above shows that for an abstract algebraic lamination $L$ on $F_r$ the property of $\beta_{L_{\Gamma},\mathfrak J}(n)$ being a subexponential function does not depend on the choice of a point $T\in \cvr$ and the quotient marked metric graph $\Gamma=T/F_r$ with the induced metric structure $\mathfrak J$:

\begin{defn}[Subexponentially growing algebraic laminations]
Let $L$ be an abstract algebraic lamination on $F_r$. We say that $L$ is \emph{subexponentially growing} if for some (equivalently, every) $T\in \cvr$ for the the quotient marked metric graph $\Gamma=T/F_r$ with the induced metric structure $\mathfrak J$ the function $\beta_{L_{\Gamma},\mathfrak J}(n)$ is subexponential.
\end{defn}

\section{Hausdorff and packing dimensions}\label{sec:hdim}

We briefly recall the basic definitions related to Hausdorff and packing dimensions. We refer the reader to \cite{BP}, \cite[Section~6]{Edgar}, \cite[Ch. 2]{Falconer} and \cite{How} for a more detailed background info.  

\begin{defn}[Hausdorff measure]
Let $(X,d)$ be a metric space and let $S\subseteq X$ be a subset of $X$. Let $\delta\ge 0$ be arbitrary

For $\epsilon>0$ put
\[
H_{\delta,\epsilon}(S)=\inf\{ \sum_{i\in I} ({\rm diam}(U_i))^\delta : S\subseteq \cup_{i\in I} U_i, \, {\rm diam}(U_i)<\epsilon, \, I \text{ countable}\}
\]
Thus the infimum is taken over all countable covers of $S$ by sets of diameter $<\epsilon$.
For $S$ and fixed $\delta\ge 0$, the function $H_{\delta,\epsilon}(S)$ is monotone non-increasing in $\epsilon>0$. Therefore the following (possibly infinite) limit exists:

\[
H_{\delta}(S):=\lim_{\epsilon\to 0+} H_{\delta,\epsilon}(S)=\sup_{\epsilon>0} H_{\delta,\epsilon}(S).
\]
The number $H_{\delta}(S)$ is called the \emph{$\delta$-dimensional Hausdorff measure} of $S$ or just  \emph{$\delta$-Hausdorff measure} of $S$.
\end{defn}

It is known that $H_\delta$ is a metric outer measure on $(X,d)$ and that Borel subsets of $X$ are $H_\delta$-measurable.

For a fixed $S$, the definition above directly implies that $H_{\delta}(S)$ is a monotone non-increasing function of $\delta$. Moreover, it turns out that if $S\subseteq X$, there exists a unique number $c\in [0,\infty]$ such that for all $0\le \delta <c$ we have $H_{\delta}(S)=\infty$ and for all $\delta>c$ we have $H_{\delta}(S)=0$. This number $c$ is called the  \emph{Hausdorff dimension} of $S$. More precisely:

\begin{defn}[Hausdorff dimension]
Let $(X,d)$ be a metric space and let $S\subseteq X$ be a subset of $X$.

The \emph{Hausdorff dimension} $\dim_H(S)$ is defined as
\[
\dim_H(S):=\inf \{\delta\ge 0: H_\delta(S)=0\}
\]
where this infimum is interpreted as $\infty$ if $H_\delta(S)>0$ for all $\delta\ge 0$.
\end{defn}

\begin{defn}[Packing measures]
Let $(X,d)$ be a metric space and let $S\subseteq X$ be a subset of $X$.

For $s\ge 0$ the \emph{$s$-packing pre-measure} of $S$ is 
\[
P_{0,s}(S)=\limsup_{\delta\to 0+} \{ \sum_{i\in I} diam(B_i)^s \}
\]
where $(B_i)_{i\in I}$ are all countable covers of $S$ by disjoint closed balls of diameter $\le \delta$ and centers in $S$.

For $s\ge 0$ the \emph{$s$-packing measure of $S$} is 
\[
P_s(S)=\inf\{\sum_{j\in J} P_{0,s}(S_j): S\subseteq \cup_{j\in J} S_j, \text{ where $J$ is countable} \}.
\]
\end{defn}
It turns out that, for $s\ge 0$,  $P_{0,s}$ is in general only a pre-measure on $X$, but $P_s$ is an actual metric outer measure on $(X,d)$. 

For a given subset $S\subseteq X$, the family $(P_s(S))_{s\ge 0}$ exhibits a similar behavior to Hausdorff measures, with a unique critical value:

\begin{defn}[Packing dimension]
Let $(X,d)$ be a metric space and let $S\subseteq X$ be a subset of $X$.
The \emph{packing dimension} of $S$ is
\[
\dim_p(S):=\inf\{s\ge 0: P_s(S)=0\}
\]
where again this infimum is taken to by $+\infty$ if $P_s(S)>0$ for all $s\ge 0$.
\end{defn}

Recall that for metric spaces $(X,d_X)$ and $(Y,d_Y)$ and a subset $S\subseteq Y$, a function $f:S\to Y$ is \emph{$\alpha$-H\"older} (where $\alpha>0$) if there exists $C>0$ such that for all $x,x'\in S$
\[
d_Y(f(x),f(x'))\le C d_X(x,x')^\alpha
\]
A $1$-H\"older function is called \emph{Lipschitz}.

We note several basic properties of Hausdorff and packing dimensions:

\begin{prop}\label{prop:dim}\cite{BP,Edgar}
Let $(X,d_X)$ and $(Y,d_Y)$ be metric spaces. Let $S\subseteq X$ and let $f:S\to Y$ be a function.

\begin{enumerate}
\item~\cite[(2.7.3)]{BP} We have $\dim_H(S)\le dim_p(S)$.
\item If $Q\subseteq S$ then $\dim_H(Q)\le \dim_H(S)$ and $\dim_p(Q)\le \dim_p(S)$.
\item If $f$ is a a Lipschitz function then
\[
\dim_H(f(S))\le \dim_H(S)
\]
and 
\[
\dim_p(f(S))\le \dim_p(S)
\]
In particular, if $\dim_H(S)=0$ then $\dim_H(f(S))=0$, and similarly, if $\dim_p(S)=0$ then $\dim_p(f(S))=0$.
\item~\cite[Proposition~6.8.9]{Edgar} If $f:S\to Y$ is $\alpha$-H\"older then
\[
\dim_H(f(S))\le \frac{1}{\alpha}\dim_H(S)
\]
and
\[
\dim_p(f(S))\le \frac{1}{\alpha}\dim_p(S).
\]

\item If $f:S\to Y$ is a H\"older function and $\dim_H(S)=0$ then $\dim_H(f(S))=0$.
\item If $f:S\to Y$ is a H\"older function and $\dim_p(S)=0$ then $\dim_p(f(S))=0$.
\item~\cite[Section 2.10]{Fe} If $A=\cup_{j\in J} A_j$ for subsets of a metric space $(X,d)$, where $J$ is countable, then 
\[
\dim_H(A)=\sup_{j\in J} \dim_H A_j.
\]

\item~\cite[Ch 2.7]{BP} If $A=\cup_{j\in J} A_j$ for subsets of a metric space $(X,d)$, where $J$ is countable, then 
\[
\dim_p(A)=\sup_{j\in J} \dim_p A_j.
\]
\end{enumerate}
\end{prop}

We also need the following fact regarding the behavior of the above dimensions under products:

\begin{prop}\label{prop:How}\cite{How}
Let $(A,d_A)$ and $(B,d_B)$ be metric spaces and let $A\times B$ be equipped with the product metric $d$.
Then

\[
\begin{aligned}
\dim_H(A)+\dim_H(B)&\le \dim_H(A\times B)\le \dim_H(A)+\dim_p(B)\\
&\le \dim_p(A\times B)\le \dim_p(A)+\dim_p(B).
\end{aligned}
\]
 
\end{prop}

We will not explicitly use the above definition of the packing dimension, but we need the following estimate for it.

For a subset $S$ of a metric space $(X,d)$ and a number $\delta>0$, denote by $N(S,\delta)$ the smallest number of closed balls of diameter $\le \delta$ needed to cover $S$. (And we put $N(S,\delta)=\infty$ if no such finite cover exists.)

\begin{prop}\label{prop:pack}
Let $(X,d)$ be a metric space and let $S\subseteq X$ be a subset of $X$. 

Then

\[
\dim_p(S)\le \limsup_{\delta\to 0+} \frac{\log N(S,\delta)}{\log(1/\delta)}. 
\]
\end{prop}
\begin{proof}
We have $\dim_p(S)\le \overline\dim_B(S)$ (see (2.7.3) in \cite{BP}), where $\overline\dim_B(S)$ is the \emph{upper box dimension} of $S$. In turn, $\overline\dim_B(S)$ is defined as
\[
\overline\dim_B(S)= \limsup_{\delta\to 0+} \frac{\log N(S,\delta)}{\log(1/\delta)}. 
\]
The statement of the proposition now follows.
\end{proof}

\section{Subexponentially growing laminations and packing dimension}

\subsection{Standard markings}

\begin{conv}
In this subsection we fix a free basis $A=\{a_1,\dots, a_r\}$ for $F_r=F(A)$ (where $r\ge 2$). We also fix the marked metric graph $R_A$ which is a rose of $r$ directed loop-edges labeled by $a_1,\dots, a_r$, where every edge has length $1$.  Then $T_A=\widetilde R_A$, again with all edges of length $1$ is exactly the Cayley tree of $F_r$ with respect to the free basis $A$.

We also fix an algebraic lamination $L$ on $F_r$ and consider the one-sided laminary subshift  $X_{L,R_A}^+$.  The subshift $X_{L,R_A}^+$ was originally defined as a subshift of the full shift $(A\cup A^{-1})^\N$. However, $X_{L,R_A}^+$ is also a subshift of the subshift $\partial_A F_r$ of $(A\cup A^{-1})^\N$ consisting of all semi-infinite freely reduced words over $A^{\pm 1}$. Note that $\partial_A F_r$ is canonically identified with the hyperbolic boundary of $F_r$,  $\partial_A F_r=\partial F_r$, where we view elements of $\partial_A F_r$ as geodesic rays in $T_A$ starting at $1\in F_r$.  

We metrize $\partial_A F_r$ using the standard visual ultrametric with the visual parameter $a=e$, namely $d(p,q)=e^{-n}$ where $n\ge 0$ is the length of the maximal common initial segment of geodesic rays $p,q\in \partial_A F_r$. Recall that at the level of languages over $A^{\pm 1}$ we have $\mathcal F[X_{L,R_A}^+]=L_{R_A}$.
\end{conv}

Recall that topological entropy of $X_{L,R_A}^+$ and of $L_{R_A}$ is
\[
h(X_{L,R_A}^+)=h(L_{R_A})=\lim_{n\to \infty} \frac{\log p_{L_{R_A}}(n)}{n},
\]
where $p_{L_{R_A}}(n)$ is the complexity function for $L_{R_A}$,

\begin{prop}\label{prop:FeHu}.
Let $L$ be an algebraic lamination on $F_r$.

Then

\[
\dim_H(X_{L,R_A}^+)=\dim_p(X_{L,R_A}^+)=h(L_{R_A}).
\]
\end{prop}
\begin{proof}

Note that the shift-map $\theta_A: \partial_A F_r\to\partial F_r$, erasing the first letter of every geodesic ray, is continuous and in fact Lipshitz, and the space $(\partial_A F_r,d)$ is compact. Moreover, by construction, for every $p\in X_{L,R_A}^+$ we have $\theta_A(p)\in X_{L,R_A}^+$, so that $X_{L,R_A}^+$ is closed $\theta_A$-invariant subset of $\partial_A F_r$.

The statement of the proposition now follows directly from Example~1.4 and Proposition~2.1(4) in \cite{FeHu}.

(The conclusion of the proposition also follows from a classic result of Furstenberg, \cite[Proposition~III.1]{Fur}, once that result is translated into modern terminology.)
\end{proof}

\begin{cor}\label{cor:FeHu} Let $L$ be an algebraic lamination on $F_r$. Let $\mathcal E_L\subseteq \partial F_r$ be the set of endpoints of $L$.
Then:
\begin{enumerate}
\item We have
\[
\mathcal E_L=\cup_{g\in F_r} gX_{L,R_A}^+.
\]
\item We have
\[
\dim_H(X_{L,R_A}^+)=\dim_p(X_{L,R_A}^+)=h(L_{R_A}).
\]
\end{enumerate}
\end{cor}

\begin{proof}
Part (1) follows from the definitions of $\mathcal E_L$ and $X_{L,R_A}^+\subseteq \partial_A F_r$.

Note that for every $g\in F_r$ the left translation $t_g: \partial_A F_r\to \partial_A F_r, p\mapsto gp$ is a bi-Lipshitz equivalence. Hence for every $g\in F_r$ we have
$\dim_p(gX_{L,R_A}^+)=\dim_p(X_{L,R_A}^+)$ and $\dim_p(gX_{L,R_A}^+)=\dim_p(X_{L,R_A}^+)$.

Part (2) now follows from part (1) and Proposition~\ref{prop:FeHu} since $F_r$ is countable and by Proposition~\ref{prop:dim} the Hausdorff dimension and the packing dimension are countably stable. Indeed,
\[
\dim_p(\mathcal E_L)=\dim_p(\cup_{g\in F_r} gX_{L,R_A}^+)=\sup_{g\in F_r} \dim_p(X_{L,R_A}^+)=h(L_{R_A})
\]
and
\[
\dim_H(\mathcal E_L)=\dim_H(\cup_{g\in F_r} gX_{L,R_A}^+)=\sup_{g\in F_r} \dim_H(X_{L,R_A}^+)=h(L_{R_A}).
\]
\end{proof}

\subsection{General case}

\begin{cor}\label{cor:var}
Let $\Gamma$ be an algebraic lamination on $F_r$. Then the following conditions are equivalent:
\begin{enumerate}
\item For some $T\in \cvr$ and the corresponding marked graph $\Gamma=T/F_r$ the laminary language $L_{\Gamma}$ is subexponentially growing.
\item For every $T\in \cvr$ and the corresponding marked metric graph $\Gamma=T/F_r$ the laminary language $L_\Gamma$ is subexponentially growing.
\item For some $T\in \cvr$ and the corresponding marked graph $\Gamma=T/F_r$ the laminary language $L_{\Gamma}$ has $h(L_{\Gamma})=0$.
\item For every $T\in \cvr$ and the corresponding marked graph $\Gamma=T/F_r$ the laminary language $L_{\Gamma}$ has $h(L_{\Gamma})=0$.
\item For some visual metric $d$ on $\partial F_r$ we have $\dim_H(\mathcal E_L)=0$.
\item For every visual metric $d$ on $\partial F_r$ we have $\dim_H(\mathcal E_L)=0$.
\item For some visual metric $d$ on $\partial F_r$ we have $\dim_p(\mathcal E_L)=0$.
\item For every visual metric $d$ on $\partial F_r$ we have $\dim_p(\mathcal E_L)=0$.
\item For some visual metric $d$ on $\partial F_r$ and the corresponding product metric $\hat d$ on $\partial F_r\times \partial F_r$ we have $\dim_H(\mathcal E_L\times \mathcal E_L)=0$.
\item For every visual metric $d$ on $\partial F_r$ and the corresponding product metric $\hat d$ on $\partial F_r\times \partial F_r$ we have $\dim_H(\mathcal E_L\times \mathcal E_L)=0$.
\item For some visual metric $d$ on $\partial F_r$ and the corresponding product metric $\hat d$ on $\partial F_r\times \partial F_r$ we have $\dim_p(\mathcal E_L\times \mathcal E_L)=0$.
\item For every visual metric $d$ on $\partial F_r$ and the corresponding product metric $\hat d$ on $\partial F_r\times \partial F_r$ we have $\dim_p(\mathcal E_L\times \mathcal E_L)=0$.
\item For some visual metric $d$ on $\partial F_r$ and the corresponding product metric $\hat d$ on $\partial F_r\times \partial F_r$ we have $\dim_H(L)=0$.
\item For every visual metric $d$ on $\partial F_r$ and the corresponding product metric $\hat d$ on $\partial F_r\times \partial F_r$ we have $\dim_H(L)=0$.
\item For some visual metric $d$ on $\partial F_r$ and the corresponding product metric $\hat d$ on $\partial F_r\times \partial F_r$ we have $\dim_p(L)=0$.
\item For every visual metric $d$ on $\partial F_r$ and the corresponding product metric $\hat d$ on $\partial F_r\times \partial F_r$ we have $\dim_p(L)=0$.
\end{enumerate}
\end{cor}
\begin{proof}
We have $(1) \iff (2)$ by Proposition~\ref{prop:general}. Since by Proposition~\ref{prop:dim} the properties of having Hausdorff dimension and having packing dimension $0$ are bi-H\"older invariants, we have $(5) \iff (6)$, $(7) \iff (8)$, $(9) \iff (10)$, $(11) \iff (12)$, $(13) \iff (14)$, $(15) \iff (16)$.

Since by Proposition~\ref{prop:dim}, $\dim_H\le \dim_p$, we have $(15)\Rightarrow (13)$, $(11)\Rightarrow (9)$, $(7)\Rightarrow (5)$. 

By Proposition~\ref{prop:How} we have $(7)\Rightarrow (11)$, $(7)\Rightarrow  (15)$, $(7)\Rightarrow  (9)$, $(7)\Rightarrow  (13)$.
Since $\mathcal L$ is the image of $L$ and of $\mathcal E_L\times \mathcal E_L$ under the projection on the first coordinate, which is a Lipschitz map, we have $(11)\Rightarrow (7)$, $(15)\Rightarrow (7)$, $(9)\Rightarrow (5)$ and $(13)\Rightarrow (5)$.

Since $\mathcal L\subseteq \mathcal E_L\times \mathcal E_L$, we have $(9)\Rightarrow (13)$ and $(11)\Rightarrow (15)$.

Remark~\ref{rem:ent} implies that $(1)\iff (3)$ and $(2)\iff (4)$.

Suppose now that (3) holds. As noted above, it follows that (1) holds and hence (2) holds, and hence (4) holds as well. Thus, in particular $h(L_{R_A})=0$ for the standard marking $R_A$ corresponding to a free basis $A$ of $F_r$.

Then by Corollary~\ref{cor:FeHu},
\[
\dim_H(\mathcal E_{R_A})=\dim_p(\mathcal E_{R_A})=h(L_{R_A})=0,
\]
so that (5) and (7) hold as well.

Suppose now that (5) holds. Since Hausdorff dimension is a bi-H\"older invariant, it follows that (6) holds. In particular, for the standard marked graph $R_A$ we have $\dim_H(\mathcal E_{R_A})=0$. Then Corollary~\ref{cor:FeHu} implies that $\dim_p(\mathcal E_{R_A})=h(L_{R_A})=0$, so that (3) holds. This completes the proof of all the required equivalences.

\end{proof}

We now provide a different, more direct and intuitive proof, without invoking the results of \cite{FeHu}, of the portion of Corollary~\ref{cor:var} that is most relevant to us in this paper

\begin{conv}
For the remainder of this section, unless specified otherwise, let $r\ge 2$ and let $L\subseteq \partial^2 F_r$ be an abstract algebraic lamination on $F_r$ of subexponential growth. Let $\mathcal E_L\subseteq \partial F_r$ be the set of endpoints of the leaves of $L$.
Let $T\in\cvr$ be any tree in the (non-projectivized) Outer space such that all edges of $T$ have length $1$.  Let $a>1$ be an arbitrary visual parameter and let $d_a$ be the corresponding standard visual metric on $\partial T=\partial F_r$ with respect to some base-vertex $x_0\in T$. Also, let $\Gamma=T/F_r$ be the quotient metric graph with the induced metric graph structure $\mathfrak J$.

For a point $x\in T, x\ne x_0$ denote by $Cyl_T[x_0,x]\subseteq \partial T$ the set of all $p\in \partial T$ such that the geodesic ray $[x_0,p)_T$ in $T$ has $[x_0,x]$ as its initial segment.

%Also, for points $x,y\in T, x\ne y$, denote by $\widehat{Cyl}_T[x,y]\subseteq \partial^2 T=\partial^2 F_r$ the set of all $(p,q)\in \partial^2 T$ such that the bi-infinite geodesic $(p,q)_T$ from $p$ to $q$ in $T$ contains $[x,y]_T$ as a subpath. 

For the visual metric $d_a$ on $\partial T$ as above, denote by $\hat d_a$ the corresponding product metric on $\partial T\times \partial T$. 
\end{conv}

The assumption above that all edges of $T$ and of $(\Gamma, \mathfrak J)$ have length $1$ is made to simplify the dimension computations. Since all visual metrics on $\partial F_r$ (and the corresponding product metrics on $\partial F_r\times \partial F_r$) are bi-H\"older equivalent, the properties for a subset to have $0$ packing dimension or $0$ Hausdorff dimension do not depend on the choice of $T\in \cvr$. 

\begin{defn}
Let $x\in T$ be an arbitrary vertex. 

Denote by $\mathcal E_L[x_0,x]$ the set of all $p\in \mathcal E_L$ such that there exists a bi-infinite geodesic $\tau$ in $T$ from another point $q\in \mathcal E_L$ to $p$ with the property that $(q,p)\in L$, that $x\in \tau$ and that the geodesic ray $[x_0,p)_T$ has the form $[x_0,p)_T=[x_0,x]_T\tau_+$ where $\tau_+$ is the subray of $\tau$ from $x$ to $p$.
\end{defn}

Note that in the context of the above definition the concatenated path $[x_0,x]\tau_+$ is path-reduced as written.

\begin{rem}\label{rem:union}
Note that $\mathcal E_L=\cup_{x\in VT} \mathcal E_L[x_0,x]$. Indeed, if $p\in \mathcal E_L$ then there exists a bi-infinite geodesic $\tau$ in $T$ from another point $q\in \mathcal E_L$ to $p$. Put $x\in \tau$ to be the closest point on $\tau$ to $x_0$. Then $x\in \tau$ is a vertex and $p\in \mathcal E_L[x_0,x]$.
\end{rem}

\begin{prop}\label{prop:HD}
Let $x\in VT$ be a vertex. Then
\[
\dim_p({\mathcal E}_L[x_0,x])=0
\]

\end{prop}

\begin{proof}
Let $0<\epsilon<1$ be arbitrary.

Since $\beta_{L_\Gamma,\mathfrak J}(n)$ is subexponential, there exists $n_0\ge 1$ such that for all $n\ge n_0$ we have 
\[
\beta_{L_\Gamma,\mathfrak J}(n+1)\le a^{n\epsilon}.
\]
For every $\delta>0$ there exists a unique integer $n=n(\delta)\ge 1$ such that $\frac{1}{a^{n+1}}\le \delta < \frac{1}{a^{n}}$. Thus 
\[
n\log a\le \log(1/\delta)\le (n+1)\log a.
\]

Consider an arbitrary $\delta>0$ which is small enough so that $n=n(\delta)\ge 2d_T(x_0,x)+n_0+3$.

 Let $y_1,\dots, y_t$ be the set of all the distinct points (which are necessarily vertices) in $T$ at distance $n+1$ from $x_0$ belonging to geodesic rays from $x_0$ to points of ${\mathcal E}_L[x_0,x]$. By definition of ${\mathcal E}_L[x_0,x]$, each $[x_0,y_i]$ has the form $[x_0,x][x,y_i]$ where $[x,y_i]$ is a subsegment of some bi-infinite geodesic in $T$ with the pair of endpoints in $L$. 

Then for $i=1,\dots, t$ the projection of $[x,y_i]$ to $\Gamma$ in an edge-path that belongs to the laminary language $L_\Gamma$ and the $d_T$-length of this path is $\le n$, and these paths are distinct for $i=1,\dots, t$. Therefore $t\le \beta_{L_\Gamma,\mathfrak J}(n+1)$.

Put $U_i=Cyl_T[x_0,y_i]\subseteq \partial T=\partial F_r$, where $1\le i\le t$. Then $U_i$ is a closed (and open) ball in $(\partial F_r, d_a)$ with
\[
diam(U_i)=a^{-d_T(x_0,y_i)}=a^{-(n+1)}\le \delta.
\]

By construction, ${\mathcal E}_L[x_0,x]\subseteq U_1\cup \dots U_t$.  

Therefore
\[
N({\mathcal E}_L[x_0,x], \delta)\le t\le \beta_{L_\Gamma,\mathfrak J}(n+1)\le a^{n\epsilon}.
\]

Hence, by Proposition~\ref{prop:pack},

\begin{gather*}
\dim_p({\mathcal E}_L[x_0,x])\le \limsup_{\delta\to 0} \frac{\log N({\mathcal E}_L[x_0,x], \delta)}{\log (1/\delta)}\le \\
\le \limsup_{n\to\infty} \frac{ \log (a^{n\epsilon})} {n\log a}=\limsup_{n\to\infty} \frac{n\epsilon \log a}{n\log a}=\epsilon.
\end{gather*}

Since $0<\epsilon<1$ was arbitrary, it follows that
\[
dim_p({\mathcal E}_L[x_0,x])=0,
\]
as claimed..
\end{proof}

\begin{rem}\label{rem:dimB}
As we noted earlier, for a subset $A$ of a metric space $(X,d)$ the \emph{upper box dimension} $\overline{dim}_B(A)$ is defined as
\[
\overline{\dim}_B(A)=\limsup_{\delta\to 0+} \frac{\log N(A,
\delta)}{\log (1/\delta)}.
\]
Hence the proof of Proposition~\ref{prop:HD} actually shows that if $L$ is a subexponentially growing algebraic lamination on $F_r$ and $T,x_0$ are as in Proposition~\ref{prop:HD}, 
then for every vertex $x\in VT$ we have
\[
\overline{\dim}_B({\mathcal E}_L[x_0,x])=0.
\]
However, unlike the Hausdorff and packing dimensions, the upper box dimension is not countably stable. Therefore we cannot conclude from the above proof that $\overline{\dim}_B({\mathcal E}_L)=0$, and moreover, such a conclusion is false in general. Indeed, if $A$ is a bounded subset of a metric space $(X,d)$ then $\overline{\dim}_B(A)=\overline{\dim}_B(\overline A)$, where $\overline A$ is the closure of $A$ in $X$; see \cite[Ch 1.1]{BP}. If $L$ is a nonempty algebraic lamination on $F_r$ then $\mathcal E_L$ is a dense subset of $\partial F_r$, e.g. because $\mathcal E_L\subseteq \partial F_r$ is $F_r$-invariant and the action of $F_r$ on $\partial F_r$ is topologically minimal~\cite{KB02}. Thus $\overline {\mathcal E_L}=\partial F_r$ and hence  $\overline{\dim}_B({\mathcal E}_L)=\overline{\dim}_B(\partial F_r)$. However, $\overline{\dim}_B(\partial F_r)>0$ for all visual metrics on $\partial F_r$.
\end{rem}

\begin{cor}\label{cor:HD}
The following hold:
\begin{enumerate}
\item For the metric space $(\partial T, d_a)$ we have
\[
\dim_p({\mathcal E}_L)=0.
\]
\item For the metric space $(\partial T\times \partial T, \hat d_a)$ we have
\[
\dim_p(L)=0.
\]

\end{enumerate}
\end{cor}
\begin{proof}
 By Remark~\ref{rem:union}, we have $\mathcal E_L=\cup_{x\in VT} \mathcal E_L[x_0,x]$. Since by Proposition~\ref{prop:dim} the packing dimension is countably stable, Proposition~\ref{prop:HD} implies that
\[
\dim_p({\mathcal E}_L)=\dim_p\left( \cup_{x\in VT} \mathcal E_L[x_0,x] \right)= \max_{v\in VT} \dim_p({\mathcal E}_L[x_0,x])=0.
\]
This (1) holds.. 

Since $L\subseteq {\mathcal E}_L\times {\mathcal E}_L$, by Proposition~\ref{prop:How} we have

\[
\dim_p(L)\le \dim_p( {\mathcal E}_L\times {\mathcal E}_L)\le \dim_p({\mathcal E}_L)+ \dim_p({\mathcal E}_L)=0
\]
and hence $\dim_p(L)=0$, so that (2) holds.
\end{proof}

\begin{thm}\label{thm:Dim}
Let $a>1$ be arbitrary, let $d$ be a visual metric on $\partial F_r$ with the visual parameter $a$ and let $\hat d$ be the corresponding product metric on $\partial F_r\times \partial F_r$.

Let $L$ be a subexponentially growing algebraic lamination on $F_r$.

Then the following hold:
\begin{enumerate}
\item For the metric space $(\partial F_r, d)$ we have
\[
\dim_H(\mathcal E_L)=\dim_p(\mathcal E_L)=0.
\]
\item For the metric space $(\partial F_r\times \partial F_r, \hat d)$ we have
\[
\dim_H(L)=\dim_p(L)=0.
\]

\end{enumerate}
\end{thm}
\begin{proof}
The metric $d$ on $\partial F_r$ is bi-H\"older equivalent to the metric $d_a$ on $\partial T=\partial F_r$ as in Corollary~\ref{cor:HD}, and similarly $\hat d$ is  bi-H\"older equivalent to the metric $\hat d_a$. Since by Proposition~\ref{prop:dim} the property of a set having packing dimension $0$ is a bi-H\"older invariant, Corollary~\ref{cor:HD} implies that  for the metrics $d$ and $\hat d$ we have $\dim_p(\mathcal E_L)=0$ and $\dim_p(L)=0$.

Then, by Proposition~\ref{prop:dim},

\[
0\le \dim_H(\mathcal E_L)\le \dim_p(\mathcal E_L)=0,
\]
which implies that $\dim_H(\mathcal E_L)=0$.
Similarly,
\[
0\le \dim_H(L)\le \dim_p(L)=0,
\]
which implies that $\dim_H(L)=0$.

Thus the conclusion of the theorem is verified.

(Note that the conclusion of the theorem also follows from Corollary~\ref{cor:var} which is a more general result.)
\end{proof}

\section{Free group automorphisms and train track maps}

Recall that an element $\phi\in\Out(F_r)$ is called \emph{fully irreducible} if there do not exist $k\ge 1$ and a proper free factor $1\lneq U\lneq F_r$ of $F_r$ such that $\phi^k([U])=[U]$, where $[U]$ is the conjugacy class of $U$ in $F_r$. Also, $\phi\in\Out(F_r)$ is called \emph{atoroidal} if there do not exist $k\ge 1$ and $1\ne u\in F_r$ such that $\phi^k([u])=[u]$.

We denote
\[
\Out_{exp}(F_r)=\{\psi\in \Out(F_r): \psi \text{ is exponentially growing}\},\]
\[
\Out_{f.i.}(F_r)=\{\psi\in \Out(F_r): \psi \text{ is fully irreducible}\}.
\]

\subsection{Train track maps and train track representatives}

\begin{defn}[Graph map]
For graphs $\Delta_1,\Delta_2$ a continuous function $f:\Delta_1\to\Delta_2$ is called a \emph{graph map} if $f(V\Delta_1)\subseteq V\Delta_2$ and $f$ maps every $e\in E\Delta_1$ to an edge-path of positive length in $\Delta_2$.\footnote{This requirement means that $e$ can be subdivided into finitely many open intervals where $f$ maps the subdivision points to vertices of $\Delta_2$ and $f$ maps the subdivision open intervals homeomorphically to open 1-cells of $\Delta_2$. These homeomorphisms are often assumed to have some additional properties to avoid pathological dynamics for the cases where $\Delta_1=\Delta_2$, but such extra properties are usually suppressed in the literature as typically only the combinatorics of the map $f$ is relevant. For the discussion of more technical aspects of the topic we refer to reader to \cite{DKL15}.} A graph map $f:\Delta_1\to\Delta_2$ is \emph{tight} if for every $e\in E\Delta_1$ the edge-path $f(e)$ is immersed.
\end{defn}

Note that a composition of graph maps is again a graph map, but a composition of tight graph maps need not be tight.

\begin{defn}[Train track map]
A graph map $f:\Delta\to\Delta$ is a \emph{train track map} if for all $k\ge 1$ the map $f^k:\Delta\to\Delta$ is tight.
\end{defn}
At this level of generality we don't require $\Delta$ to be connected or finite and we don't require $f$ to be a homotopy equivalence in the above definition.

\begin{defn}[Topological and train track representatives]
Let $\phi\in\Out(F_r)$.
\begin{enumerate}
\item A \emph{topological representative} of $\phi$ is a marked graph $\Gamma$ (with $\pi_1(\Gamma)=F_r$ via a marking isomorphism) together with a tight graph map $f:\Gamma\to\Gamma$ such that the map $f$ is a homotopy equivalence and the induced map $f_\#:\pi_1(\Gamma)\to\pi_1(\Gamma)$\footnote{We don't require $f$ to have a fixed vertex or even a fixed point in $\Gamma$. Therefore  $f_\#:\pi_1(\Gamma)\to\pi_1(\Gamma)$ is only well defined as an element of $\Out(\pi_1(\Gamma))$.} is equal to $\phi$ in $\Out(F_r)$ after the identification  $\pi_1(\Gamma)=F_r$.
\item A \emph{train track representative} of $\phi$  is a topological representative $f:\Gamma\to\Gamma$ of $\phi$ such that $f$ is a train track map.
\end{enumerate}
\end{defn}
The above definition implies that if $f:\Gamma\to\Gamma$ is a train track representative of $\phi\in\Out(F_r)$ the for all $k\ge 1$ the map $f^k:\Gamma\to\Gamma$ is a train track representative of $\phi^k$.

Let $f:\Gamma\to\Gamma$ be a train track representative of some $\phi\in \Out(F_r)$. Choose an orientation $E\Gamma=E_+\Gamma\sqcup E_-\Gamma$ on $\Gamma$ and let $E_+\Gamma=\{e_1,\dots, a_m\}$. The \emph{transition matrix} $A(f)$ of $f$ is an $m\times m$ integer matrix where for $1\le i,j\le m$ the entry $a_{ij}$ is the number of occurrences of $e_{i}^{\pm 1}$ in $f(e_j)$. With this definition one has $A(f^k)=(A(f))^k$ for all $k\ge 1$.

We say that $f$ is \emph{irreducible} if  the transition matrix $A(f)$ is irreducible, that is for all $1\le i,j\le m$ there exists $k\ge 1$ such that $[A^k(f)]_{ij}>0$ (or, equivalently, if for all $1\le i,j\le m$ there exists $k\ge 1$ such that $e_i$ or $e_i^{-1}$ occurs in $f^k(e_j)$). W say that $f$ is \emph{primitive} if the matrix $A(f)$ is primitive, that is, there exists $k\ge 1$ such that for all $1\le i,j\le m$  we have $[A^k(f)]_{ij}>0$. Finally, $f$ is \emph{expanding} if for every edge $e\in E\Gamma$ we have $\lim_{k\to\infty} |f^k(e)|=\infty$ (or, equivalently, if for every edge $e\in E\Gamma$ there exists $k\ge 1$ such that $|f^k(e)|\ge 2$). For a train track map $f:\Gamma\to\Gamma$ representing $\phi\in \Out(F_r)$ we denote by $\lambda(f)$ the spectral radius of the matrix $A(f)$. Note that the properties of the matrix $A(f)$ being irreducible or primitive and the value of $\lambda(f)$ do not depend on the choice of an orientation on $\Gamma$ and of the ordering on $E_+\Gamma$ but only depend on the map $f$ itself.

If $f:\Gamma\to\Gamma$ is an expanding irreducible train track representative of $\phi\in \Out(F_r)$ then the Perron-Frobenius theory applies to the matrix $A(f)$. It is known that in this case $\lambda(f)>1$ is a simple eigenvalue of $A(f)$ which does not depend on the choice of a train track representative $f$ of $\phi$ but depends only on the element $\phi\in\Out(F_r)$. In this case the number $\lambda(f)$ is called the \emph{stretch factor} of $\phi$ and denoted $\lambda(\phi)$.

We record some key facts about train track representatives of fully irreducible elements of $\Out(F_r)$ that are of relevance for this paper:

\begin{prop}\label{prop:use}
Let $\phi\in \Out(F_r)$ be fully irreducible. Then the following hold:
\begin{enumerate}
\item \cite[Theorem 1.7]{BH92} There exists a train track representative $f:\Gamma\to \Gamma$ of $\phi$.
\item \cite[Lemma 2.4]{Ka14} For any train track representative $f:\Gamma\to \Gamma$ of $\phi$, the map $f$ is expanding and irreducible and, moreover, the matrix $A(f)$ is primitive.
\end{enumerate}
\end{prop}

We also need the following notion introduced and explored in detail in \cite{DGT23}

\begin{defn}
Let $\phi\in \Out(F_r)$ be fully irreducible and let $f:\Gamma\to\Gamma$ be a train track representative of $\phi$ (so that $f$ is expanding and $A(f)$ is primitive, by Proposition~\ref{prop:use}). Then, as shown in \cite{DGT23}, exactly one of the following occurs:
\begin{enumerate}
\item There exists an orientation $E\Gamma=E_+\Gamma\sqcup E_-\Gamma$, called a \emph{preferred orientation}, such that for every $e\in E_+\Gamma$ the path $f(e)$ contains only edges from $E_+\Gamma$ (and therefore, by inversion, for every $e\in E_-\Gamma$, the path $f(e)$ contains only edges from $E_-\Gamma$).
\item There exist $e,e'\in E\Gamma$  such that both $e$ and $e^{-1}$  occur in $f(e')$.
\end{enumerate}
We say that $f$ is \emph{orientable} if (1) occurs and \emph{non-orientable} if (2) occurs. It is also proved in \cite{DGT23} that whether a train track representative $f$ of a fully irreducible $\phi\in \Out(F_r)$ is orientable or non-orientable depends only on $\phi$ and not on the choice of $f$. Therefore, for a fully irreducible $\phi\in \Out(F_r)$, we say that $\phi$ is \emph{orientable} (respectively, \emph{non-orientable}) if some (equivalently, any) train track representative $f$ of $\phi$ is orientable (respectively, non-orientable).
\end{defn}

The following statement is a straightforward consequence of the definitions:

\begin{lem}\label{lem:orient}
Let $\phi\in \Out(F_r)$ be fully irreducible and let $f:\Gamma\to\Gamma$ be a train track representative of $\phi$. Then:
\begin{enumerate}
\item If $f$ is non-orientable then, since $A(f)$ is primitive, there exists $k\ge 1$ such that for all $e,e'\in E\Gamma$ both $e$ and $e^{-1}$  occur in $f^k(e')$.
\item If $f$ is orientable and $E\Gamma=E_+\Gamma\sqcup E_-\Gamma$ is a preferred orientation, then again, since $A(f)$ is primitive, there exists $k\ge 1$ such that for all $e,e'\in E_+\Gamma$ the edge $e$ occurs in $f^k(e')$. Moreover, in this case for all $n\ge 1$ and all $e,e'\in E_+\Gamma$, the entry $[A(f^n)]_{e,e'}=[A^n(f)]_{e,e'}$ is equal to the number of occurrences of $e$ in $f^n(e')$.
\item In case (1) the language $\mathcal F[\theta_f]$ is closed under taking inverses, that is, $\overline{\mathcal F[\theta_f]}=\mathcal F[\theta_f]$.
\item In case (2), we have $\gamma\in \overline{\mathcal F[\theta_f]}$ if and only if there exist $k\ge 1$ and $e\in E_-\Gamma$ such that that $\gamma$ is a nontrivial subpath of $f^k(e)$.
\end{enumerate}
\end{lem}

Part (1) and (2) of Lemma~\ref{lem:orient} directly imply:

\begin{prop}\label{prop:cases}
Let $f:\Gamma\to\Gamma$ be a train track representative of a fully irreducible element $\phi\in\Out(F_r)$.
\begin{enumerate}
\item If $f$ is non-orientable, then $f$ extends to a substitution $\theta_f: (E\Gamma)^\ast\to (E\Gamma)^\ast$ and this substitution is primitive.
\item If $f$ is orientable and $E\Gamma=E_+\Gamma\sqcup E_-\Gamma$ is a preferred orientation, then $f$ extends to a substitution $\theta_f: (E_+\Gamma)^\ast\to (E_+\Gamma)^\ast$ and this substitution is primitive.
\end{enumerate}
\end{prop}

\subsection{Attracting laminations of fully irreducibles}\label{sect:L}

The following notion was introduced by Bestvina, Feighn and Handel in \cite{BFH97}.
\begin{defn}[Attracting lamination of a fully irreducible automorphism]
Let $\phi\in\Out(F_r)$ be a fully irreducible element and let $f:\Gamma\to\Gamma$ be a train track representative of $\phi$. Put $T_f=\tilde\Gamma\in\cvr$ where every edge of $T_f$ is given length $1$.  The \emph{attracting lamination} $L_f$ of $f$ is an abstract algebraic lamination $L_f\subseteq \partial^2 F_r$ on $F_r$ such that for distinct $p,q\in \partial F_r$ we have $(p,q)\in L_f$ if and only if for every finite edge subpath $\beta$ of positive length in the bi-infinite geodesic $\tau$ from $p$ to $q$ in $T_f$, the projection $\gamma$ of $\beta$ to $\Gamma$ satisfies the property that there exist $k\ge 1$ and $e\in E_\Gamma$ such that $\gamma$ is a subpath of $f^k(e)$.

It is proved in \cite{BFH97} that $L_f$ depends only on $\phi$ and not on the choice of a train track representative $f$ of $\phi$. Therefore we also call $L_f$ the \emph{attracting lamination} of $\phi$ and denote  $L_\phi:=L_f$.
\end{defn}

\begin{prop}\label{prop:key}
Let $f:\Gamma\to\Gamma$ be a train track representative of a fully irreducible element $\phi\in\Out(F_r)$.

\begin{enumerate}
\item If $f$ is non-orientable, then $(L_f)_\Gamma=\mathcal F[\theta_f]$.
\item If $f$ is orientable and $E\Gamma=E_+\Gamma\sqcup E_-\Gamma$ is a preferred orientation, then $$(L_f)_\Gamma=\mathcal F[\theta_f]\sqcup \overline{\mathcal F[\theta_f]}.$$
\item $L_f$ is a subexponentially growing lamination on $F_r$. More specifically, there exists $C_0>0$ such that for all $n\ge 0$ we have $\beta_{(L_f)_\Gamma}(n)\le C_0n^2$.
\end{enumerate}

\end{prop}
\begin{proof}
Parts (1) and (2) follows by comparing the definitions of $L_f$ and $\theta_f$ and using Proposition~\ref{prop:cases}.

If $f$ is non-orientable, then by part (1) $(L_f)_\Gamma=\mathcal F[\theta_f]$. By Proposition~\ref{prop:cases} the substitution  $\theta_f: (E\Gamma)^\ast\to (E\Gamma)^\ast$ is primitive. Hence the laminary language $(L_f)_\Gamma=\mathcal F[\theta_f]$ is subexponentially growing by Corollary~\ref{cor:primsubst}.

Suppose $f$ is orientable and $E\Gamma=E_+\Gamma\sqcup E_-\Gamma$ is a preferred orientation. By Proposition~\ref{prop:cases} the substitution  $\theta_f: (E_+\Gamma)^\ast\to (E_+\Gamma)^\ast$ is primitive, and therefore, by Corollary~\ref{cor:primsubst}, the language $\mathcal F[\theta_f]$ is subexponentially growing, and more specifically, $\beta_{\mathcal F[\theta_f]}(n)\le Cn^2$ for all $n\ge 1$ where $C\ge 1$ is some constant. Part (2) of Proposition~\ref{prop:key} implies that
\[
\beta_{(L_f)_\Gamma}(n)=2\beta_{\mathcal F[\theta_f]}(n)\le 2Cn^2
\]
for all $n\ge 1$, so that  the laminary language $(L_f)_\Gamma$ is subexponentially growing.

Thus (3) is verified in both cases, as required.
\end{proof}

\begin{thm}\label{thm:HD-fi}
Let $\phi\in\Out(F_r)$ (wher $r\ge 2$) be fully irreducible and let $L_\phi\subseteq \partial^2 F_r$ be the attracting lamination of $\phi$. Let $\mathcal E_{L_\phi}\subseteq \partial F_r$ be the set of endpoints of $L_\phi$.

Let $a>1$ be an arbitrary visual parameter, let $d$ be a corresponding visual metric on $\partial T=\partial F_r$ and let $\hat d$ be the product metric on $\partial F_r\times \partial F_r$.

 Then:
\begin{enumerate}
\item For the metric space $(\partial T, d)$ we have
\[
\dim_H(\mathcal E_{L_\phi})=\dim_p(\mathcal E_{L_\phi})=0.
\]
\item For the metric space $(\partial F_r\times \partial F_r, \hat d)$ we have
\[
\dim_H(L_\phi)=\dim_p(L_\phi)=0.
\]

\item For the metric space $(\partial F_r, d)$ we have
\[
\dim_H\left( \bigcup_{\psi\in\Out_{f.i.}(F_r)} \mathcal E_{L_\psi} \right)=\dim_p\left( \bigcup_{\psi\in\Out_{f.i.}(F_r)} \mathcal E_{L_\psi} \right)=0.
\]
\item For the metric space $(\partial^2 F_r, \hat d)$ we have
\[
\dim_H\left( \bigcup_{\psi\in\Out_{f.i.}(F_r)} L_\psi \right)=\dim_p\left( \bigcup_{\psi\in\Out_{f.i.}(F_r)} L_\psi \right)=0.
\]

\end{enumerate}

\end{thm}
\begin{proof}
Since by Proposition~\ref{prop:key}, $L_\phi$ is a subexponentially growing abstract algebraic lamination on $F_r$. Therefore the conclusion of parts (1) and (2) of the theorem follows by Theorem~\ref{thm:Dim}.

Now parts (1) and (2) directly imply parts (3) and (4) since the index set $\Out_{f.i.}(F_r)$ is countable and the Hausdorff and packing dimensions are countably stable by Proposition~\ref{prop:dim}.
\end{proof}

\subsection{Exponentially growing automorphisms}\label{sect:LL}

Recall that an element $\phi\in Out(F_r)$ is \emph{exponentially growing} if for some (equivalently, any) free basis $A$ of $F_r$ there exists $1\ne w\in F_r$ such that
\[
\lambda_A(\phi, w)=\limsup_{n\to\infty}\sqrt[n]{||\phi^n(w)||_A} >1.
\]
In particular, every fully irreducible $\phi\in Out(F_r)$ is exponentially growing.
We denote
\[
\Out_{exp}(F_r)=\{\psi\in \Out(F_r): \psi \text{ is exponentially growing}\},
\]

For an arbitrary exponentially growing $\phi\in Out(F_r)$ the definition of the set $\mathcal L(\phi)$ of attracting laminations for $\phi$, originally introduced in \cite{BFH00} is considerably more complicated than in the fully irreducible case and we omit its details here. (Note, however, that $\mathcal L(\phi)=\mathcal L(\phi^k)$ for any $k\ge 1$.) The cleanest modern version uses  ``completely split" (relative) train track representatives $f:\Gamma\to\Gamma$, introduced in~\cite{FH11}, which always exist once an exponentially growing $\phi\in Out(F_r)$ is replaced by its suitable positive power.

If $f:\Gamma\to\Gamma$ a completely split train track representing an  exponentially growing $\phi\in Out(F_r)$, then for every $L\in \mathcal L(\phi)$ there exists a fixed vertex $v$ of $\Gamma$ and a fixed direction $e\in E\Gamma$ from an exponentially growing stratum such that for the ``combinatorial eigenray" $\rho_e=e\dots$ in $\Gamma$ (defined similarly to the case of a primitive substitution or a fully irreducible $\phi$) one has

\[
L_\Gamma=\mathcal F[\rho_e],
\]
up to a possible symmetrization.

In \cite{HL22} Hilion and Levitt classify the complexity types associated expanding fixed points of $\partial F_r$ for elements of $\Aut(F_r)$ and use this classification to prove (see Corolloary~5.10 in \cite{HL22}:

\begin{prop}\label{prop:HL22}
Let $\phi\in \Out(F_r)$ be exponentially growing, let $L\in \mathcal L(\phi)$ and let $(\alpha,\Gamma)$ be a marked graph structure for $F_r$.
Then $p_\Gamma(n)$ is equivalent to one of $n, n\log n, n\log\log n, n^2$.
\end{prop}

\begin{thm}\label{thm:HD-exp}
Let $\phi\in\Out(F_r)$ (wher $r\ge 2$) be exponentially growing and let $L\in \mathcal L(\phi)$, $L\subseteq \partial^2 F_r$ be an attracting lamination of $\phi$. Let $\mathcal E_{L}\subseteq \partial F_r$ be the set of endpoints of $L$.

Let $a>1$ be an arbitrary visual parameter, let $d$ be a corresponding visual metric on $\partial F_r$ and let $\hat d$ be the product metric on $\partial F_r\times\partial F_r$.

 Then:
\begin{enumerate}
\item For the metric space $(\partial F_r, d)$ we have
\[
\dim_H(\mathcal E_{L})=\dim_p(\mathcal E_{L})=0.
\]
\item For the metric space $(\partial F_r\times \partial F_r, \hat d)$ we have
\[
\dim_H(L)=\dim_p(L)=0.
\]

\item For the metric space $(\partial F_r, d)$ we have

\[
\dim_H\left( \bigcup_{\psi\in\Out_{exp}(F_r)} \bigcup_{L\in \mathcal L(\psi)} \mathcal E_{L} \right)=\dim_p\left( \bigcup_{\psi\in\Out_{exp}(F_r)} \bigcup_{L\in \mathcal L(\psi)} \mathcal E_{L} \right)=0.
\]

\item For the metric space $(\partial F_r\times \partial F_r, \hat d)$ we have

\[
\dim_H\left( \bigcup_{\psi\in\Out_{exp}(F_r)} \bigcup_{L\in \mathcal L(\psi)} L \right)=\dim_p\left( \bigcup_{\psi\in\Out_{exp}(F_r)} \bigcup_{L\in \mathcal L(\psi)} L \right)=0
\]

\end{enumerate}

\end{thm}
\begin{proof}
Proposition~\ref{prop:HL22} implies that $L$ is a subexponentially growing abstract algebraic lamination on $F_r$. 
Therefore the conclusion of parts (1) and (2) of the theorem follows by Theorem~\ref{thm:Dim}.

Now parts (1) and (2) directly imply parts (3) and (4) since since the index set $\Out_{exp}(F_r)$ is countable and for every exponentially growing $\psi\in \Out(F_r)$ the set $\mathcal L(\psi)$ is finite.
\end{proof}

\section{Ending laminations and the Cannon-Thurston map}\label{sect:EL}

Given $\phi\in \Out(F_r)$ and an automorphism $\Phi\in \Aut(F_r)$ in the outer automorphism class $\phi$, one can consider the free-by-cyclic group
\[
G=F_r\rtimes_\Phi \mathbb Z=F_r\rtimes_\Phi \langle t\rangle=\langle F_r, t| tht^{-1}=\Phi(h), h\in F_r\rangle. \tag{$\dag$}
\]
Replacing the stable letter $t$ in the above HNN-extension by $t_1=ut$, where $u\in F_r$, replaces $\Phi$ by its composition with the inner automorphism of $F_r$ corresponding to the conjugation by $u$. Therefore  $G$ only depends on $\phi$ rather than on $\Phi$ and we denote $G_\phi=G$. The group $G_\phi$ naturally fits into a short exact sequence
\[
1\to F_r \to G_\phi \to \mathbb Z\to 1.\tag{$\clubsuit$}
\]

Brinkmann~\cite{Br00} proved that $G_\phi$ is word-hyperbolic if and only if $\phi$ is atoroidal. A general result of Mitra~\cite{M98a} about short exact sequences of word-hyperbolic groups implies that if $G_\phi$ is word-hyperbolic then the inclusion $i:F_r\to G_\phi$ extends to a continuous $F_r$-equivariant map $\partial i: \partial F_r\to \partial G_\phi$ (which in this case is easily seen to be surjective). The map $\partial i$ is called the \emph{Cannon-Thurston map}.

In \cite{M97} Mitra also describes the fibers of $\partial i$ for the case of short exact sequences of three word-hyperbolic groups. We will recall how Mitra's result works for the case of a short exact sequence $(\clubsuit)$ above, assuming that $\phi\in\Out(F_r)$ is atoroidal (so that $G_\phi$ is word-hyperbolic).

\begin{defn}[Ending lamination]
Let $\phi\in \Out(F_r)$ be an atoroidal element. Let $f:\Gamma\to\Gamma$ be a topological representative of $\phi$ and let $T_f=\tilde \Gamma\in \cvr$, where every edge of $T_f$ is given length $1$.

For $1\ne h\in F_r$ put $\Lambda_{\phi,h}\subseteq \partial^2 F_r$ to be the set of all $(p,q)\in \partial^2 F_r$ such that for every finite edge subpath $\beta$ of positive length in the bi-infinite geodesic $\tau$ from $p$ to $q$ in $T_f$, the projection $\gamma$ of $\beta$ to $\Gamma$ has the property that for some $n\ge 1$ $\gamma$ is a subpath of a cyclic permutation of the cyclicallay tightened form of the path $f^n(h)$ in $\Gamma$. (The assumption that $\phi$ is atoroidal implies that $n\to\infty$ as $|\beta|\to\infty$.)
Put
\[
\Lambda_\phi=\bigcup_{1\ne h\in F_r} \Lambda_{\phi,h}.
\]
It is known~\cite{KL15,DKT16} that $\Lambda_\phi$ is an abstract algebraic lamination on $F_r$ and that its definition above does not depend on the choice of a topological representative of $\phi$. See \cite{DKT16} for an equivalent description of $\Lambda_\phi$ directly in terms of $\partial F_r$, without the reference to choosing a topological representative of $\phi$ or even a generating set for $F_r$.

The subset $\Lambda_\phi\subseteq \partial^2 F_r$ is called the \emph{ending lamination corresponding to $\phi$}.
\end{defn}

A key result of Mitra~\cite{M97} about the fibers of the Cannon-Thurston map for short exact sequences of word-hyperbolic groups, when applied to $G_\phi$, implies:

\begin{prop}\label{prop:Mitra}\cite{M97}
Let $\phi\in \Out(F_r)$ be an atoroidal element and let $\partial i: \partial F_r\to \partial G_\phi$ be the Cannon-Thurston map.
Then for distinct $p,q\in \partial F_r$ we have $\partial i(p)=\partial i(q)$ if and only if $(p,q)\in \Lambda_{\phi} \cup \Lambda_{\phi^{-1}}$.
\end{prop}

For a subset $S\subseteq \partial^2 F_r$, denote by $Tran(S)\subseteq \partial^2 F_r$ the transitive closure of $S$ in $\partial^2 F_r$. That is, $Tran(S)$ consists of all $(p,q)\in \partial^2 F_r$ such that for some $n\ge 1$ there exist $(p_0,p_1), (p_1,p_2), \dots, (p_{n-1},p_n)\in S$ with $p_0=p$ and $p_n=q$.

In \cite{KL15} Kapovich and Lustig showed how $\Lambda_\phi$ is related to $L_\phi$ in the case where $\phi\in\Out(F_r)$ is atoroidal and fully irreducible:

\begin{prop}\label{prop:KL}\cite{KL15}
Let $\phi\in\Out(F_r)$ be atoroidal and fully irreducible.

Then $\Lambda_\phi=Tran(L_\phi)=L(T_-(\phi))$, where $T_-(\phi)\in \overline{\cvr}$ is the attracting tree for $\phi^{-1}$ and where $L(T_-(\phi))$ is the dual lamination of $T_-(\phi)$, in the sense of \cite{CHL08b}.  Moreover, $\Lambda_\phi\setminus L_\phi$ consists of $F_N$-orbits of finitely many points of $\partial^2 F_r$.
\end{prop}

%\begin{cor}\label{cor:tran}
%Let $\phi\in \Out(F_r)$ be be an atoroidal element. Then $\mathcal E_{\Lambda_\phi}=\mathcal E_{L_\phi}$.
%\end{cor}
 %\begin{proof}
 %By Proposition~\ref{prop:KL}, $\Lambda_\phi=Tran(L_\phi)$ which directly implies $\mathcal E_{\Lambda_\phi}=\mathcal E_{L_\phi}$ by the definition of transitive closure.
 %\end{proof}

 \begin{cor}\label{cor:HD-fi}
 Let $\phi\in \Out(F_r)$ is atoroidal and fully irreducible and let $G_\phi=F_r\rtimes_\phi \mathbb Z$ be the corresponding free-by-cyclic group. Let $\Lambda_\phi\subseteq \partial^2 F_r$ be the ending algebraic lamination of $\phi$. Let $\Lambda\subseteq \partial^2 F_r$ be the Cannon-Thurston lamination for the Cannon-Thurston map corresponding to the inclusion $F_r\le G_\phi$.

Let $a>1$ be an arbitrary visual parameter and let $d$ be a visual metric on $\partial F_r$ with the visual parameter $a$. Let $\hat d$ be the corresponding product metric on $\partial F_r\times \partial F_r$.
 Then:
\begin{enumerate}
\item For the metric spaces $(\partial F_r, d)$ and $(\partial F_r\times \partial F_r, \hat d)$  we have
\[
\dim_H(\mathcal E_{\Lambda_\phi})=\dim_p(\mathcal E_{\Lambda_\phi})=0 \text{ and  } \dim_H(\mathcal E_{\Lambda})=\dim_p(\mathcal E_{\Lambda})=0
\]
and 
\[
\dim_H(\Lambda_\phi)=\dim_p(\Lambda_\phi)=0 \text{ and  } \dim_H(\Lambda)=\dim_p(\Lambda)=0.
\]
\item For the metric space $(\partial F_r, d)$ and $(\partial F_r\times \partial F_r, \hat d)$ we have
\[
\dim_H\left( \bigcup_{\psi\in\Out_{f.i.a.}(F_r)} \mathcal E_{\Lambda_\psi} \right)=\dim_p\left( \bigcup_{\psi\in\Out_{f.i.a.}(F_r)} \mathcal E_{\Lambda_\psi} \right)=0.
\]
and
\[
\dim_H\left( \bigcup_{\psi\in\Out_{f.i.a.}(F_r)} \Lambda_\psi \right)=\dim_p\left( \bigcup_{\psi\in\Out_{f.i.a.}(F_r)} \Lambda_\psi \right)=0.
\]
\end{enumerate}
\end{cor}

\begin{proof}

For (1), note that, for an atoroidal fully irreducible $\phi$, by Proposition~\ref{prop:KL} we have $\Lambda_\phi=Tran(L_\phi)$ which directly implies $\mathcal E_{\Lambda_\phi}=\mathcal E_{L_\phi}$ by the definition of transitive closure.
Then Theorem~\ref{thm:HD-fi} implies that $\mathcal E_{\Lambda_\phi}$ has Hausdorff dimension $\dim_H(\mathcal E_{\Lambda_\phi})=0$. Hence, by equivalences established in Corollary~\ref{cor:var}. we have 
\[
\dim_H(\mathcal E_{\Lambda_\phi})=\dim_p(\mathcal E_{\Lambda_\phi})=\dim_H(\Lambda_\phi)=\dim_p(\Lambda_\phi)=0.
\]
The same conclusion applies to $\phi^{-1}$. Recall that by Mitra~\cite{M97}, we have $\Lambda=\Lambda_{\phi}\cup \Lambda_{\phi^{-1}}$. 
Therefore
\[
\dim_H(\mathcal E_{\Lambda})=\dim_p(\mathcal E_{\Lambda})= \dim_H(\Lambda)=\dim_p(\Lambda)=0,
\]
and (1) holds, as required.

Now part (1) implies part (2) since the index set $\Out_{f.i.a.}(F_r)$ is countable.

\end{proof}

\section{Open problems}\label{sect:problems}

Theorem~\ref{thm:A} and Corollary~\ref{cor:B} raise several further natural questions.

\begin{prob}\label{prob:1}
Let $\phi\in\Out(F_r)$ is fully irreducible and atoroidal, and let $d$ be a visual metric on $\partial F_r$ for some $a>1$.

Let $\partial i:\partial F_r\to\partial G_\phi$ be the Cannon-Thurston map.
Let $\mathcal{N}\subseteq \partial G_\phi$ be the set of all is non-conical limit points for the action of $F_r$ on $\partial G_\phi$, and let $\mathcal B=(\partial i)^{-1}(\mathcal N)$.

Is it true that  $\dim_H(\mathcal B)=0$?

\end{prob}

A result of \cite{JKLO} provides a characterization of non-conical limit points for a Cannon-Thurston map $j:\partial H \to Z$ for the convergence action of a non-elementary word-hyperbolic metric group $H$ on a metrizable compactum $Z$. Applied to the setting of Problem~\ref{prob:1}, this result says that $p\in \partial F_r$ belongs to $\mathcal B$ if and only if $p$ is "asymptotic" to the Cannon-Thurston lamination $\Lambda=\Lambda_\phi\cup \Lambda_{\phi^{-1}}$. There are both dynamical and geometric characterizations of being "asymptotic" to $\Lambda$ obtained in \cite{JKLO}.  Using these results, \cite{JKLO}  also proves that if the Cannon-Thurston map $\partial i: \partial H\to\partial G$ is non-injective, where $H$ is a non-elementary word-hyperbolic subgroup of a word-hyperbolic group $G$, then there always exists a non-conical limit point $z\in \partial G$ for the action of $H$ on $\partial G$ such that $\#((\partial i)^{-1}(z)=1$. In the context of Problem~\ref{prob:1} it follows that $\Lambda=\Lambda_\phi\cup \Lambda_{\phi^{-1}}$ is a proper subset of $\mathcal B$. However, precisely how much bigger than $\Lambda_\phi\cup \Lambda_{\phi^{-1}}$ the set $\mathcal B$ is remains unclear, particularly in terms of the Hausdorff dimension of $\mathcal B$.

As noted above, the group $G_\phi=F_r\rtimes_\phi\mathbb Z$ is word-hyperbolic if and only if $\phi$ is atoroidal, which includes a wide variety of situations where $\phi$ is atoroidal but not fully irreducible. In light of Corollary~\ref{cor:B} it is natural to ask:

\begin{prob}
Let $\phi\in\Out(F_r)$ be an arbitrary atoroidal element. Does the conclusion of Corollary~\ref{cor:B} hold for the set of endpoints $\mathcal E_{\Lambda_\phi}\subseteq \partial F_r$ of the ending lamination $\Lambda_\phi$? Is it true that the abstract algebraic lamination $\Lambda_\phi$ is subexponentially growing?
\end{prob}

Kapovich and Lustig~\cite{KL15} proved that if $\phi\in \Out(F_r)$ is atoroidal and fully irreducible then $\Lambda_\phi=L(T_{\phi^{-1}})$ where $L(T_{\phi^{-1}})$ is the dual algebraic lamination (in the sense of \cite{CHL08b}) of the attracting tree $T_{\phi^{-1}}\in\overline{cvr}$ of the fully irreducible $\phi^{-1}$. Here $\overline{\cvr}$ is the closure of $\cvr$ in the equivariant Gromov-Hausdorff convergence topology, or, equivalently, in the length function topology~\cite{Pa}.

Attracting trees of fully irreducibles provide  basic examples of "arational" trees in $\overline{cvr}$. Recall that an $\mathbb R$-tree $T\in \overline{\cvr}$ is called \emph{arational} if for its dual lamination $L(T)$ no leaf of $L(T)$ is carried by a proper free factor of $F_r$. Bestvina and Reynolds proved~\cite{BR15} that the hyperbolic boundary $\partial \mathcal{FF}_r$ of the free factor complex $\mathcal{FF}_r$ is equal to $\mathcal{AT}_r/\sim$ where $\mathcal{AT}_r\subseteq  \overline{\cvr}$ is the set of all arational trees and where for $T,S\in \mathcal{AT}_r$ we have $T\sim S$ whenever $L(T)=L(S)$. This naturally raises the following question:

\begin{prob}
Let $T\in \mathcal{AT}_r$ be arbitrary. Does the conclusion of Theorem~\ref{thm:A} hold for $L(T)$? Is it true that the abstract algebraic lamination $L(T)$ is subexponentially growing?
\end{prob}
Note that the work of Coulbois and Hilion~\cite{CH14}, for $T\in\overline{\cvr}$, provides a description of $L(T)$, in terms of a certain dynamical system given by a finite collection of partial isometries of a compact subtree $K$ of $T$. It may be possible to use this description for  $T\in \mathcal{AT}_r$, perhaps under some additional ergodicity assumptions, to gain information about the growth rate of the laminary language of $L(T)$.

Another situation where the structure of the Cannon-Turson laminations on $F_r$ is well understood concerns extensions of $F_r$ by purely atoroidal convex-cocompact subgroups of of $Out(F_r)$. Every such subgroup $Q\le \Out(F_r)$ is itself word-hyperbolic and defines an extension $1\to F_r \to G_Q\to Q\to 1$ where $G_Q$ is known to be again word-hyperbolic by a result of Dowdall and Taylor~\cite{DT18}. A general result of Mitra for short exact sequences of hyperbolic groups then implies that the Cannon-Thurston lamination $\Lambda$ has the form $\Lambda=\cup_{z\in \partial Q}\Lambda_z$ where $\Lambda_z\subseteq \partial^2 F_r$ are ending algebraic laminations on $F_r$. In \cite{DKT16} Dowdall, Kapovich and Taylor showed that for each $z\in \partial Q$ there is an $F_r$-free arational tree $T_z\in \overline{\cvr}$ with $L(T_z)=\Lambda_z$. Thus $\Lambda=\cup_{z\in \partial Q} L(T_z)$. This fact is used in \cite{DKT16} to show that in this situation the Cannon-Thurston map $\partial i: \partial F_r\to \partial G_Q$ is at most $2r$-to-one.

\begin{prob}
Let $Q\le \Out(F_r)$ be a purely atoroidal convex-cocompact subgroup defining the extension $G_Q$ of $F_r$ by $Q$. Let $\Lambda=\cup_{z\in \partial Q}\Lambda_z=\cup_{z\in \partial Q} L(T_z)$ be the corresponding Cannon-Thurston lamination on $F_r$.

(a) If  $d$ is a visual metric on $\partial F_r$, is it true that $\dim_H(\Lambda)=0$?

(b) Let $\mathcal B\subseteq \partial F_r$ be the full $\partial i$-preimage of the set of non-conical limit points in $\partial G_Q$ (for the action of $F_r$ on $\partial G_Q$). Is it true that $\dim_H(\mathcal B)=0$?
\end{prob}

The most general versions of the problems discussed above concern the context of Cannon-Thurston maps for arbitrary word-hyperbolic groups:

\begin{prob}
Let $H$ be a non-elementary word-hyperbolic group with a convergence action on a metrizable compactum $Z$ such that the Cannon-Thurston map $\partial i:\partial H\to Z$ exists and let $\Lambda\subseteq \partial^2 H$ be the associated Cannon-Thurston algebraic lamination. Let $d$ be a visual metric on $\partial H=\partial X$ coming from a properly discontinuous and cocompact isometric action of $H$ on some Gromov-hyperbolic space $X$.

(a) For the metric space $(\partial H, d)$, is it true that $\dim_H(\mathcal E_\Lambda)=0$ and $\dim_H(\Lambda)=0$?

(b) More generally, if $\mathcal B\subseteq \partial H$ is the full $\partial i$-preimage of the set of all non-conical limit points in $Z$, is it true that $\dim_H(\mathcal B)=0$?
\end{prob}
Note that since all visual metrics on $\partial H$ are H\"older-equivalent, the answer to the above problem does not depend on the choice of $X$ and of the visual metric $d$.


\begin{thebibliography}{999}

%\bibitem{An}
%A. Ancona,
%\emph{Positive harmonic functions and hyperbolicity.} Potential theory -- surveys and problems (Prague, 1987), 1--23,
%Lecture Notes in Math., 1344, Springer, Berlin, 1988



\bibitem{BirSer}
J. S. Birman and C. Series, \emph{Geodesics with bounded intersection number on surfaces are sparsely distributed}, Topology \textbf{24} (1985), no. 2, 217--225.
\href{https://doi.org/10.1016/0040-9383(85)90056-4}{doi:10.1016/0040-9383(85)90056-4}


\bibitem{B}
M. Bestvina, \emph{Geometry of outer space.} Geometric group theory, 173--206, IAS/Park City Math. Ser., 21, Amer. Math. Soc., Providence, RI, 2014



\bibitem{BFH97}
M. Bestvina, M. Feighn and M. Handel, \emph{Laminations, trees, and irreducible automorphisms of free groups}, Geom. Funct. Anal. \textbf{7} (1997), 215--244.
\href{https://doi.org/10.1007/PL00001618}{doi:10.1007/PL00001618}


\bibitem{BFH00}
M. Bestvina, M. Feighn and M. Handel, \emph{The Tits alternative for Out(Fn). I. Dynamics of exponentially-growing automorphisms.}
Ann. of Math. (2) \textbf{151} (2000), no. 2, 517--623.
\href{https://doi.org/10.2307/121043}{doi:10.2307/121043}



\bibitem{BH92}
M. Bestvina and M. Handel,
\emph{Train tracks and automorphisms of free groups.}
Ann. of Math. (2) \textbf{135} (1992), no. 1, 1--51.
\href{https://doi.org/10.2307/2946562}{doi:10.2307/2946562}


\bibitem{BR15}
M. Bestvina and P. Reynolds,
\emph{The boundary of the complex of free factors.}
Duke Math. J. \textbf{164} (2015), no. 11, 2213--2251.
\href{https://doi.org/10.1215/00127094-3129702}{doi:10.1215/00127094-3129702}



\bibitem{BP}
C. J. Bishop and Y. Peres,
\emph{Fractals in probability and analysis.}
Cambridge Studies in Advanced Mathematics, 162,  Cambridge University Press, Cambridge, 2017.  ISBN: 978-1-107-13411-9

\bibitem{Br00}
P. Brinkmann, \emph{Hyperbolic automorphisms of free groups}, Geom. Funct. Anal. \textbf{10} (2000), no. 5, 1071--1089.
\href{https://doi.org/10.1007/PL00001647}{doi:10.1007/PL00001647}



\bibitem{CT}
J. W. Cannon and W. P. Thurston, \emph{Group invariant Peano curves},  Geom. Topol. \textbf{11} (2007), 1315--1355.
\href{https://doi.org/10.2140/gt.2007.11.1315}{doi:10.2140/gt.2007.11.1315}


\bibitem{Coor}
M. Coornaert, \emph{Mesures de Patterson-Sullivan sur le bord d'un espace hyperbolique au sens de Gromov.} Pacific J. Math. \textbf{159} (1993), no. 2, 241--270.
\href{https://doi.org/10.2140/pjm.1993.159.241}{doi:10.2140/pjm.1993.159.241}



\bibitem{CH14}
T. Coulbois and A. Hilion,
\emph{Rips induction: index of the dual lamination of an $\mathbb R$-tree.}
Groups Geom. Dyn. \textbf{8} (2014), no. 1, 97--134.
\href{https://doi.org/10.4171/GGD/218}{doi:10.4171/GGD/218}


\bibitem{CHL08a}
T. Coulbois, A. Hilion,  and M. Lustig, \emph{$\mathbb R$-trees and laminations for free groups. I. Algebraic laminations.} J. Lond. Math. Soc. (2) \textbf{78} (2008), no. 3, 723--736.
\href{https://doi.org/10.1112/jlms/jdn052}{doi:10.1112/jlms/jdn052}



\bibitem{CHL08b}
T. Coulbois, A. Hilion,  and M. Lustig, \emph{$\mathbb R$-trees and laminations for free groups. II. The dual lamination of an $\mathbb R$-tree.} J. Lond. Math. Soc. (2) \textbf{78} (2008), no. 3, 737--754.
\href{https://doi.org/10.1112/jlms/jdn053}{doi:10.1112/jlms/jdn053}



\bibitem{CHR15}
T. Coulbois, A. Hilion,  and P. Reynolds, \emph{Indecomposable $F_N$-trees and minimal laminations.} Groups Geom. Dyn. \textbf{9} (2015), no. 2, 567--597.
\href{https://doi.org/10.4171/GGD/321}{doi:10.4171/GGD/321}




\bibitem{CV}
M. Culler, K. Vogtmann, \emph{Moduli of graphs and automorphisms of free groups}, Invent. Math. \textbf{84} (1986), no. 1. 91--119.
\href{https://doi.org/10.1007/BF01388734}{doi:10.1007/BF01388734}



\bibitem{DGT23}
S. Dowdall, R. Gupta and S. J. Taylor,
\emph{Orientable maps and polynomial invariants of free-by-cyclic groups.} Adv. Math. \textbf{415} (2023), Paper No. 108872


\bibitem{DKL15}
S. Dowdall, I. Kapovich,  and C. J. Leininger, \emph{Dynamics on free-by-cyclic groups.} Geom. Topol. \textbf{19} (2015), no. 5, 2801--2899.
\href{https://doi.org/10.2140/gt.2015.19.2801}{doi:10.2140/gt.2015.19.2801}


\bibitem{DKT16}
S. Dowdall, I. Kapovich,  and S. J. Taylor, \emph{Cannon-Thurston maps for hyperbolic free group extensions.}
Israel J. Math. \textbf{216} (2016), no. 2, 753--797.
\href{https://doi.org/10.1007/s11856-016-1426-2}{doi:10.1007/s11856-016-1426-2}

\bibitem{DSU17}
T. Das, D. Simmons, and M. Urbanski,
\emph{Geometry and dynamics in Gromov hyperbolic metric spaces. With an emphasis on non-proper settings.}
Mathematical Surveys and Monographs, 218.
American Mathematical Society, Providence, RI, 2017.
ISBN: 978-1-4704-3110-5.


\bibitem{DT18}
S. Dowdall and S. J. Taylor,
\emph{Hyperbolic extensions of free groups.}
Geom. Topol. \textbf{22} (2018), no. 1, 517--570.
\href{https://doi.org/10.2140/gt.2018.22.517}{doi:10.2140/gt.2018.22.517}



\bibitem{Edgar}
G. Edgar, \emph{Measure, topology, and fractal geometry. Second edition.} Undergraduate Texts in Mathematics. Springer, New York, 2008.  ISBN: 978-0-387-74748-4

\bibitem{Falconer}
K. Falconer, \emph{Fractal geometry. Mathematical foundations and applications. Third edition.} John Wiley \& Sons, Ltd., Chichester, 2014.  ISBN: 978-1-119-94239-9



\bibitem{Fe}
H. Federer,
\emph{Geometric measure theory.}
Die Grundlehren der mathematischen Wissenschaften, Band 153. Springer-Verlag New York, Inc., New York, 1969, ISBN: 978-3-540-60656-7.

\bibitem{FH11}
M. Feighn and M. Handel, \emph{The recognition theorem for $Out(F_n)$}, Groups Geom. Dyn. \textbf{5} (2011), no. 1, 39--106.
\href{https://doi.org/10.4171/GGD/116}{doi:10.4171/GGD/116}


\bibitem{FeHu}
D.-J, Feng and W. Huang, 
\emph{Variational principles for topological entropies of subsets.}
J. Funct. Anal. \textbf{263} (2012), no. 8, 2228--2254


\bibitem{FM}
S. Francaviglia and A. Martino, \emph{Metric properties of outer space}, Publ. Mat. \textbf{55} (2011), no. 2, 433--473.
\href{https://doi.org/10.5565/PUBLMAT\_55211\_09}{doi:10.5565/PUBLMAT\_55211\_09}.



\bibitem{Fur}
H. Furstenberg,
\emph{Disjointness in Ergodic Theory, Minimal Sets, and a Problem in Diophantine Approximation.}
Mathematical Systems Theory \textbf{1} (1967), 1--49.
\href{https://doi.org/10.1007/BF01692494}{doi:10.1007/BF01692494}


\bibitem{HM19}
M. Handel and L. Mosher,
\emph{The free splitting complex of a free group, II: Loxodromic outer automorphisms.} Trans. Amer. Math. Soc. \textbf{372} (2019), no. 6, 4053--4105

\bibitem{HM20}
M. Handel and L Mosher,
\emph{Subgroup decomposition in $Out(F_n)$.} Mem. Amer. Math. Soc. \textbf{264} (2020), no. 1280, ISBN: 978-1-4704-4113-5

\bibitem{HM23}
M. Handel and L. Mosher,
\emph{Hyperbolic actions and 2nd bounded cohomology of subgroups of $Out(F_n)$.} Mem. Amer. Math. Soc. \textbf{292} (2023), no. 1454,  ISBN: 978-1-4704-6698-5

\bibitem{HL22}
A. Hilion and G. Levitt,
\emph{A Pansiot-type subword complexity theorem for automorphisms of free groups,} arXiv:2208.00676


\bibitem{How}
J. D. Howroyd, 
\emph{On Hausdorff and packing dimension of product spaces.}
Math. Proc. Cambridge Philos. Soc. \textbf{119} (1996), no. 4, 715--727

\bibitem{JKLO}
W. Jeon,  I. Kapovich,  C. J. Leininger, and K. Ohshika,
\emph{Conical limit points and the Cannon-Thurston map.}
Conform. Geom. Dyn. \textbf{20} (2016), 58--80.


\bibitem{Ka14}
I. Kapovich, \emph{Algorithmic detectability of iwip automorphisms.} Bull. Lond. Math. Soc. \textbf{46} (2014), no. 2, 279--290.
\href{https://doi.org/10.1112/blms/bdt093}{doi:10.1112/blms/bdt093}



\bibitem{KB02}
I. Kapovich, and N. Benakli, \emph{Boundaries of hyperbolic groups.} Combinatorial and geometric group theory (New York, 2000/Hoboken, NJ, 2001), 39--93,
Contemp. Math., 296, Amer. Math. Soc., Providence, RI, 2002


\bibitem{KL09}
I. Kapovich and M. Lustig,
\emph{Geometric intersection number and analogues of the curve complex for free groups.} Geom. Topol. \textbf{13} (2009), no. 3, 1805--1833.
\href{https://doi.org/10.2140/gt.2009.13.1805}{doi:10.2140/gt.2009.13.1805}


\bibitem{KL10}
I. Kapovich and M. Lustig, \emph{Intersection form, laminations and currents on free groups.} Geom. Funct. Anal. \textbf{19} (2010), no. 5, 1426--146.
\href{https://doi.org/10.1007/s00039-009-0041-3}{doi:10.1007/s00039-009-0041-3}


\bibitem{KL14}
I. Kapovich and M. Lustig, \emph{Invariant laminations for irreducible automorphisms of free groups}, Q. J. Math. \textbf{65} (2014),  no. 4, 1241--1275.

\bibitem{KL15}
I. Kapovich and M. Lustig, \emph{Cannon-Thurston fibers for iwip automorphisms of $F_N$},
J. Lond. Math. Soc. (2) \textbf{91} (2015), no. 1, 203--224.
\href{https://doi.org/10.1112/jlms/jdu069}{doi:10.1112/jlms/jdu069}



\bibitem{KL20}
M. Kapovich and B. Liu, \emph{Hausdorff dimension of non-conical limit sets.} Trans. Amer. Math. Soc. \textbf{373} (2020), no. 10, 7207--7224


\bibitem{Le09}
G. Levitt, \emph{Counting growth types of automorphisms of free groups.} Geom. Funct. Anal. \textbf{19} (2009), no. 4, 1119--1146.
\href{https://doi.org/10.1007/s00039-009-0016-4}{doi:10.1007/s00039-009-0016-4}


\bibitem{Lu22}
M. Lustig, \emph{How do topological entropy and factor complexity behave under monoid morphisms and free group basis changes?} preprint, arXiv:2204.00816



\bibitem{LM18}
C. Lecuire and M. Mj,
\emph{Horospheres in degenerate $3$-manifolds.}
Int. Math. Res. Not. IMRN \textbf{2018} (2018), no. 3, 816--861.
\href{https://doi.org/10.1093/imrn/rnw188}{doi:10.1093/imrn/rnw188}.



\bibitem{LM21}
D. Lind and B. Marcus, An introduction to symbolic dynamics and coding.
Second edition.Cambridge Math. Lib.
Cambridge University Press, Cambridge, 2021. ISBN: 978-1-108-84871-8.

%\bibitem{LS}
%R. Lyndon and P. E. Schupp,
%\emph{Combinatorial group theory.} Reprint of the 1977 edition. Classics in Mathematics. Springer-Verlag, Berlin, 2001; ISBN: 3-540-41158-5

\bibitem{LS18}
A. Lenzhen and J. Souto, \emph{Variations on a theorem of Birman and Series.} Ann. Inst. Fourier (Grenoble) \textbf{68} (2018), no. 1, 171--194.
\href{https://doi.org/10.5802/aif.3156}{doi:10.5802/aif.3156}




\bibitem{M97}
M. Mitra, \emph{Ending laminations for hyperbolic group extensions}, 
Geom. Funct. Anal. \textbf{7} (1997), no. 2, 379--402.
\href{https://doi.org/10.1007/PL00001624}{doi:10.1007/PL00001624}.

\bibitem{M98a}
M. Mitra, \emph{Cannon--Thurston maps for hyperbolic group extensions}, 
Topology \textbf{37} (1998), no. 3, 527--538.
\href{https://doi.org/10.1016/S0040-9383(97)00033-4}{doi:10.1016/S0040-9383(97)00033-4}.

\bibitem{Pa}
F. Paulin, \emph{The Gromov topology on $\mathbb{R}$-trees}, 
Topology Appl. \textbf{32} (1989), no. 3, 197--221.
\href{https://doi.org/10.1016/0166-8641(89)90061-3}{doi:10.1016/0166-8641(89)90061-3}.

\bibitem{S21}
J. Sapir,
\emph{A Birman--Series type result for geodesics with infinitely many self-intersections.}
Trans. Amer. Math. Soc. \textbf{374} (2021), no. 11, 7553--7568.
\href{https://doi.org/10.1090/tran/8424}{doi:10.1090/tran/8424}.

\bibitem{Q}
M. Queff\'elec, \emph{Substitution dynamical systems--spectral analysis.} 
Second edition. Lecture Notes in Mathematics, 1294. Springer-Verlag, Berlin, 2010. ISBN: 978-3-642-11231-7.

\bibitem{V15}
K. Vogtmann, \emph{On the geometry of outer space.} 
Bull. Amer. Math. Soc. (N.S.) \textbf{52} (2015), no. 1, 27--46.
\href{https://doi.org/10.1090/S0273-0979-2014-01454-3}{doi:10.1090/S0273-0979-2014-01454-3}.







\end{thebibliography}
\end{document}